\documentclass[10pt,oneside]{amsart} 
\usepackage{amsmath}
\usepackage{amsthm}
\usepackage{amsfonts}
\usepackage{amssymb,amscd,epsf,verbatim,mathtools}
\usepackage{mathrsfs}
\usepackage{graphicx}
\usepackage{latexsym}
\usepackage{lscape}
\usepackage[colorlinks=true]{hyperref}
\hypersetup{colorlinks, citecolor=blue, filecolor=black, linkcolor=red, urlcolor=green}
\usepackage{epstopdf}
\usepackage{tikz}
\usetikzlibrary{calc}
\usetikzlibrary{matrix,arrows,decorations.pathmorphing}
\usepackage{tikz-cd}
\usepackage{color}
\usepackage{multirow}  
\usepackage{scalefnt}
\usepackage[shortlabels]{enumitem}
\usepackage{geometry}

\newcommand{\Z}{\mathbb{Z}}
\newcommand{\F}{\mathbb{F}}

\newcommand{\Q}{\mathbb{Q}}
\newcommand{\A}{\mathbb{A}}
\renewcommand{\L}{\mathbb{L}}

\newcommand{\C}{\mathbb{C}}

\newcommand{\mC}{\mathcal{C}}

\newcommand{\mH}{\mathcal{H}}

\newcommand{\mO}{\mathcal{O}}

\newcommand{\fm}{\mathfrak{m}} 
\newcommand{\fU}{\mathfrak{U}}
\newcommand{\fX}{\mathfrak{X}}

\newcommand{\cl}{\overline}

\newcommand{\so}{\Rightarrow}

\newcommand{\ra}{\rightarrow}

\newcommand{\sq}{\widetilde}

\DeclareMathOperator{\Ainf}{\text{A}_{\textup{inf}}}
\DeclareMathOperator{\Ainfx}{\mathbb{A}_{\textup{inf}, X}}
\DeclareMathOperator{\Acris}{\text{A}_{\textup{cris}}} 
\DeclareMathOperator{\Bcrisp}{\text{B}_{\textup{cris}}^+}
\DeclareMathOperator{\Bcris}{\text{B}_{\textup{cris}}}
\DeclareMathOperator{\Bdrp}{\text{B}_{\textup{dR}}^+}
\DeclareMathOperator{\Bdr}{\text{B}_{\textup{dR}}}

\DeclareMathOperator{\crys}{crys}
\DeclareMathOperator{\qSyn}{qSyn}

\DeclareMathOperator{\qrsP}{qrsPerfd}

\newcommand{\gr}[1]{\langle {#1} \rangle} 
\DeclareMathOperator{\spec}{Spec}
\DeclareMathOperator{\spf}{Spf}
\DeclareMathOperator{\spa}{Spa}
 
\DeclareMathOperator{\etale}{\textup{\'etale}}
\DeclareMathOperator{\ett}{\textup{\'et}}
\DeclareMathOperator{\pet}{\textup{pro\'et}}
\DeclareMathOperator{\isom}{\;\xrightarrow{\: {}_{\sim} \:} \;}

\newcommand{\bi}{\begin{itemize}}
\newcommand{\ei}{\end{itemize}}
\newcommand{\be}{\begin{enumerate}}
\newcommand{\ee}{\end{enumerate}}

\newtheorem{theorem}{Theorem}[section]
\newtheorem{definition}[theorem]{Definition}
\newtheorem{construction}[theorem]{Construction}
\newtheorem{lemma}[theorem]{Lemma}
\newtheorem{corollary}[theorem]{Corollary}
\newtheorem{proposition}[theorem]{Proposition}

\newtheorem{remark}[theorem]{Remark}

\newtheorem{mainthm}{Theorem} 

\newtheorem{mainlemma}[mainthm]{Lemma}
\newtheorem{maincor}[mainthm]{Corollary}

\newtheorem{mainrmk}[mainthm]{Remark}

\title[The crystalline comparison of $\Ainf$-cohomology]{\small The crystalline comparison of $\Ainf$-cohomology: \\ the case of good reduction} 
\author{Zijian Yao}
\email{zijian.yao.math@gmail.com}
\address{Department of Mathematics, Harvard University.}

\begin{document}

\begin{abstract}
We provide a simple approach for the crystalline comparison of $\Ainf$-cohomology, and reprove the comparison between crystalline and $p$-adic $\etale$ cohomology for formal schemes in the case of good reduction. 
\end{abstract}
\maketitle


\section{Introduction}

The recently developed theory of $\Ainf$-cohomology $R \Gamma_{\Ainf} (\fX)$ in \cite{BMS}, for smooth formal schemes $\fX$ over rings like $\mO_{\C_p}$, plays a central role in  integral $p$-adic Hodge theory --- the study of $\Z_p$-lattices in $p$-adic Galois representations. This paper studies the relation between the aforementioned theory and crystalline cohomology. More precisely, we provide a crystalline interpretation of $R \Gamma_{\Ainf} (\fX) \widehat \otimes^\L \Acris$, when valued on certain semiperfectoid covers of $\fX$. 
This allows us to obtain a simplified proof of the crystalline comparison of $\Ainf$-cohomology. In particular, we avoid some rather involved analysis in relative de Rham--Witt theory and the methods of ``all possible coordinates" in \cite{BMS}. Moreover, we expect our approach to generalize relatively easily to the case of semistable reduction. 


\subsection{Background}
To explain further let us fix notations. Let $C$ be a complete algebraically closed nonarchimedean  extension of $\Q_p$, with ring of integers $\mO_C$ and residue field $k$. Let $\fX$ be a smooth formal scheme over $\spf \mO_C$, with adic generic fiber $X = \fX_{C}^{\textup{ad}}$ over $\spa(C, \mO_C)$ and special fiber $\fX_k$ over $\spec k$.  We fix a compatible system of primitive $p$-power roots of unity in $\mO_C$, which defines an element $\epsilon:= (1, \zeta_p, \zeta_{p^2}, ... )$ in its tilt $\mO_C^\flat$. As usual let $\mu:= [\epsilon] - 1 \in \Ainf = W(\mO_C^\flat)$, and $\xi : = \mu/\varphi^{-1} (\mu)$. We denote by $\theta: \Ainf \twoheadrightarrow \mO_C$ the natural projection whose kernel is generated by $\xi$, and by $\vartheta: \Ainf \twoheadrightarrow W(k)$ the map induced from $\mO_C^\flat \twoheadrightarrow k$ by the Witt vector functoriality. Finally, recall that $\Acris$ is the $p$-completed PD-envelop of $\Ainf \twoheadrightarrow \mO_C$.  

There are at least three constructions of the $\Ainf$-cohomology $R \Gamma_{\Ainf} (\fX)$ of $\fX$ by now: via perfectoid spaces (\cite{BMS}), via topological Hochschild homology (\cite{THH}), and via the prismatic site (\cite{prismatic}). We are concerned with the first construction, which is defined as 
$$R \Gamma_{\Ainf} (\fX) :=  R \Gamma (\fX_{\ett}, L\eta_{\mu} R \nu_* \Ainfx) \in D(\Ainf).$$
Here $\Ainfx$ is (the derived $p$-completed) Fontaine's period sheaf $W(\widehat \mO_{X^\flat}^{+})$ on the pro-$\etale$ site $X_{\pet}$, the map $\nu: X_{\pet} \ra \fX_{\ett}$ is the natural map of sites, and $L \eta_{\mu}$ is the d\'ecalage operator (see Section 5 and 6 of \cite{BMS}). The perfectoid nature of this definition lies within $X_{\pet}$, which has a basis given by affinoid perfectoid objects $(S, S^+)$, on which $\Ainfx$ takes the value $\Ainf(S^+) = W(S^{+, \flat})$.\footnote{up to $W(\fm^\flat)$-torsion, where $\fm \subset \mO_C$ is the maximal ideal. See Lemma 5.6 of \textit{loc.cit}.} Note that $R \Gamma_{\Ainf} (\fX)$ is equipped with a Frobenius operator $\varphi$ induced from the Frobenius on $\Ainfx$, which essentially comes from the tilting equivalence. From this definition, it is not difficult to deduce the following comparison theorems of $R \Gamma_{\Ainf} (\fX)$.
\begin{itemize}[--]
\item 
 $\etale$ comparison:  
 $R \Gamma_{\Ainf} (\fX)\otimes_{\Ainf} \Ainf[\frac{1}{\mu}] \cong R \Gamma(X_{\ett}, \Z_p) \otimes_{\Z_p} \Ainf [\frac{1}{\mu}]. $ \footnote{For example it follows from the primitive comparison theorem of Scholze (Theorem 5.1 and the proof of Theorem 8.4 in \cite{Scholze}). In fact, Bhatt gives another proof of the $\etale$ comparison without using the primitive comparison theorem, for which one needs the de Rham comparison (at least in the almost sense). See Remark 8.4 in \cite{Bhatt}.}
\item de Rham comparison: 
$\gamma_{\textup{dR}}: R \Gamma_{\Ainf} (\fX) \otimes^\L_{\theta} \mO_C \isom R \Gamma_{\textup{dR}} (\fX). $ 
\end{itemize} 
Note that, since $R \Gamma_{\Ainf}(\fX)$ is derived $\xi$-complete (see the proof of Lemma \ref{lemma:invalid}), 
the de Rham comparison in particular ensures that $R \Gamma_{\Ainf}(\fX)$ takes values in perfect $\Ainf$-complexes. We will take the de Rham comparison 
as an input in our paper, which is the easier part of \cite{BMS}. It can be proven either from Theorem 8.3 of \textit{loc.cit.} or, alternatively, by the methods of \cite{Bhatt}. 

\subsection{Main results}
Let $R\Gamma_{\crys} ((\fX_{\mO_C/p})/\Acris)$ be the crystalline cohomology of $\fX_{\mO_C/p}$ over the base $\Acris$.  Let us now state our first first main result. 

\begin{mainthm} \label{mainthm:exist_map}
There exists a functorial $\varphi$-equivariant map 
$$ h_{\textup{crys}}: R\Gamma_{\crys} ((\fX_{\mO_C/p})/\Acris) \longrightarrow R \Gamma_{\Ainf}(\fX) \widehat \otimes_{\Ainf}^\L \Acris $$
which is compatible with the de Rham comparison $\gamma_{\textup{dR}}^{-1}$ after base change.
\end{mainthm}

\begin{mainrmk}
The map we construct here goes in the opposite direction of a map constructed in \cite{BMS} (by taking a certain limit over all choices of coordinates), which they show is an isomorphism. 
In fact, as we explain next, the existence of such a functorial map (in the direction that we construct) is sufficient to deduce the comparison theorems of interest. 
\end{mainrmk}

\begin{mainrmk} 
In \cite{Faltings}, Faltings constructs a map $\chi: R\Gamma_{\crys} ((\fX_{\mO_C/p})/\Acris) \ra R \Gamma(X_{\pet}, \A_{\textup{cris}})$  (in modern language) to relate crystalline and \'etale cohomology (he even shows that the latter is almost isomorphic to $R \Gamma(X_{\ett}, \Z_p) \otimes \Acris$). Faltings's map $\chi$ is essentially the composition of $h_{\textup{crys}}$ with the canonical map $\delta: R \Gamma_{\Ainf}(\fX) \widehat \otimes_{\Ainf}^\L \Acris \ra R \Gamma(X_{\pet}, \A_{\textup{cris}})$.  
It does not seem possible to reconstruct $h_{\textup{crys}}$ from $\chi$, and our approach for $\Bcris$ comparison is different from \cite{Faltings}. 
\end{mainrmk}

As a consequence we are able to reprove the following results of \cite{BMS}. 

\begin{mainthm} \label{mainthm:BKF_and_C_cris}
Let $\fX$ be a smooth proper formal scheme over $\mO_C$. 
\be[(1)]
\item The cohomology $H^i_{\Ainf} (\fX)$ of the perfect complex $R\Gamma_{\Ainf}(\fX) \in D(\Ainf)$ takes values in Breuil--Kisin--Fargues modules (see Definition \ref{def:BKF_module}). 
\item $R\Gamma_{\Ainf}(\fX) $ enjoys the following $\varphi$-equivariant comparison with the crystalline cohomology of the special fiber 
$$\gamma_{\textup{crys}}: R\Gamma_{\Ainf}(\fX) \otimes^\L_{\vartheta} W(k) \cong R \Gamma_{\textup{crys}} (\fX_k /W(k)).$$ 
\item Now suppose that $\fX_{\mO_K}$ is a smooth proper formal scheme over $\mO_K$, where $K$ is a complete discretely valued nonarchimedean extension of $\Q_p$ with residue field $k_0$. Let $C = C_K$ be a $p$-completed algebraic closure of $K$. Then there exists a $\textup{Gal}_K$, $\varphi$-equivariant isomorphism 
$$ H^i_{\textup{crys}} (\fX_{k_0}/W(k_0)) \otimes_{W(k_0)} \Bcris \isom H^i_{\ett} (\fX_{C}^{\textup{ad}}, \Z_p) \otimes_{\Z_p} \Bcris. $$
In particular,  $H^i_{\ett} (\fX_{C}^{\textup{ad}}, \Q_p)$ is a crystalline $\textup{Gal}_K$-representation. 
\ee
\end{mainthm}

Several remarks are in order. 
\begin{mainrmk}
The $\etale$ and crystalline specialization, together with the fact that $H^i_{\Ainf}(\fX)[\frac{1}{p}]$ is a Breuil--Kisin--Fargues module, implies the main result of \cite{BMS} on torsion relations. Namely, for each $n\ge 1$, we have 
$$\textup{length}_{W(k)} \big(H^i_{\textup{crys}}(\fX_k/W(k))_{\textup{tor}} /p^n \big) \ge \textup{length}_{\Z_p} \big(H^i_{\ett}(X, \Z_p)_{\textup{tor}} /p^n \big). $$ 
\end{mainrmk}

Let us briefly comment on the proof of Theorem \ref{mainthm:BKF_and_C_cris} from Theorem \ref{mainthm:exist_map}.

\begin{mainrmk} 
\bi[--]
\item First we observe that part (2) (the crystalline comparison of $\Ainf$ cohomology) is an immediate corollary by the derived Nakayama's lemma. This in particular recovers Theorem 1.8(i) of \cite{BMS}. 
\item The map $h_{\crys}$ has the desired direction for applications to $p$-adic Hodge theory -- in particular in view of the $\Bcris$ comparison -- as it implies that ``there are enough Galois invariants'' after base change to $W(k)$.  
\item Finally let us turn to part (1), which seems to be the most subtle part, since we do not claim that $h_{\crys}$ is an isomorphism (at least we want to avoid using it). For this we perform a careful induction on the degree of cohomology following \cite{Morrow}, which reduces to show that the $\Q_p$ dimension of $H^i_{\ett} (X, \Q_p)$ agrees with the $W(k)$-rank of $H^i_{\crys} (\fX_k/W(k))$, which we again deduce from the existence of $h_{\crys}$. 
\ei
\end{mainrmk} 

\begin{mainrmk}
One defect of the $\Bcris$ comparison stated as above is that we have ignored the information on filtration (after base change to $\Bdr$). To take care of the filtration compatibility of the $\Bcris$ comparison, we consider a certain infinitesimal cohomology of the generic fiber $X$ over the pro-thickening $\Bdrp \twoheadrightarrow C$, which we discuss in the end of this introduction.  
\end{mainrmk}

\subsection{The construction of $h_{\textup{crys}}$} Now we illustrate the strategy to prove Theorem \ref{mainthm:exist_map}.  By functoriality it suffices to construct the map on affine opens $\spf R$ in the $\etale$ site $\fX_{\ett}$. The basic idea is to use quasisyntomic descent to reduce to the case of quasiregular semiperfectoid covers of $R$, which are certain well-behaved quotients of perfectoid rings by ``quasiregular ideals''. The precise definition is recalled in Section \ref{sec:QRSP}, for now it suffices to think of them as certain perfectoid variants of local complete intersections. Typical examples include 
$$k [\![x^{1/p^\infty}, y^{1/p^{\infty}}]\!]/(x-y), \quad \mO_C/p = \mO_{C^\flat}/p^{\flat}, \quad  \textup{ and } \mO_C \gr{X^{1/p^\infty}}/(X).$$ 
More precisely,  let  $A\Omega_S$  be the derived $\Ainf$-cohomology of $S$, obtained as the derived $(p, \xi)$-completion of the left Kan extension of the functor 
$$R \longmapsto A \Omega_R := R \Gamma_{\Ainf}(\spf R, L\eta_\mu R \nu_* \Ainfx),$$ from the category of $p$-completions of smooth $\mO_C$-algebras to $D(\Ainf)$. By quasisyntomic descent it suffices to construct a functorial map from the crystalline cohomology $R\Gamma_{\textup{crys}}((R/p)/\Acris)$ to $A \Omega_S \widehat \otimes^\L \Acris$. This in turn follows from the following lemma. 

\begin{mainlemma}
For a quasiregular semiperfectoid $\mO_C$-algebra $S$, its derived $\Ainf$-cohomology $A \Omega_S$ is a discrete, topologically free $\Ainf$-algebra concentrated in degree $0$. Moreover, there is a natural map of rings $A \Omega_S \widehat \otimes^\L \Acris \twoheadrightarrow S/p$, which is a PD-thickening over $\Acris \ra \mO_C/p$. In particular, $A\Omega_S \twoheadrightarrow S/p$ is an object in the crystalline site $\textup{CRIS}(R/p)_{\Acris}.$
\end{mainlemma}

\subsection{The $\Bdrp$-cohomology of the generic fiber} 

As alluded to previously, to analyze the individual cohomology groups $H^*_{\Ainf} (\fX)$ further we need a variant of $h_\textup{crys}$ over $\Bdrp$. This is provided by reformulating the $\Bdrp$-cohomology of \cite{BMS} via a certain infinitesimal site. 

\begin{mainthm} \label{mainthm:exist_map_dR}  Let $X$ be a smooth proper rigid analytic variety over $C$. There is a cohomology $R\Gamma_{\textup{inf}}(X/\Bdrp)$ with $\Bdrp$ coefficients, satisfying $$R\Gamma_{\textup{inf}}(X/\Bdrp) \otimes^\L C \cong R \Gamma_{\textup{dR}}(X/C).$$ Moreover, there exists a functorial map
$$h_\textup{dR}: R\Gamma_{\textup{inf}}(X/\Bdrp) \longrightarrow R \Gamma_{\pet} (X, \mathbb{B}_{\textup{dR}}^+),$$
where $\mathbb{B}_{\textup{dR}}^+$ is the period sheaf on $X_{\pet}$ defined in \cite{Scholze}. Further suppose $X = \fX_{C}^{\textup{ad}}$ is the generic fiber of a smooth proper formal scheme $\fX$ as above, then there is a canonical isomorphism $$R \Gamma_{\Ainf} (\fX) \otimes^\L \Bdrp \isom R\Gamma_{\textup{inf}}(X/\Bdrp),$$
which identifies $H^i_{\textup{inf}}(X/\Bdrp)$ with the $\Bdrp$-lattice $H^i_{\Ainf}(\fX) \otimes \Bdrp$ in 
$H^i_{\ett} (X, \Z_p) \otimes \Bdr.$
Finally, $h_\textup{dR}$ is compatible with $h_{\textup{crys}}$ under the $\etale$ comparison 
in a suitable sense.
\end{mainthm}

Finally, from this we prove the rest of the main theorem of \cite{BMS}. 

\begin{maincor} \label{maincor:recover} Let $\fX_{\mO_K}$ be a smooth proper formal scheme over $\mO_K$ as in Theorem \ref{mainthm:BKF_and_C_cris} (3).  
\be
\item If $H^*_{\crys}(\fX_k/W(k))$ is torsion free for both $* = i$ and $i +1$, then the integral crystalline cohomology $H^i_{\crys}(\fX_k)$ with its Frobenius $\phi$ can be recovered from the $\textup{Gal}_K$-module $H^i_{\ett}(\fX_C^{\textup{ad}}, \Z_p).$
\item The $\Bcris$ comparison in Theorem \ref{mainthm:BKF_and_C_cris} 
is compatible with filtrations after base change to $\Bdr$. In other words,  $\textup{D}_{\textup{cris}} \big(H^i_{\ett} (\fX_{C}^{\textup{ad}}, \Q_p) \big)$ is given by $\Big(H^i_{\crys}(\fX_k/W(k))_\Q, \phi, \textup{Fil}\Big)$ as a filtered $\varphi$-module, where $\textup{Fil}$ is the Hodge filtration on $H^i_{\textup{dR}}(X_{K}/K)$. 
\ee
\end{maincor}
 
 
\subsection*{Conventions} 
We use the language of $\infty$-categories. In particular, the derived categories considered here are the natural $\infty$-categorical enhancements and algebra objects in derived categories are $E_\infty$-algebras. Moreover we do not distinguish derived categories of sheaves and sheaves valued in derived categories. 
For a category $\mC$, we denote by $\mC^{\circ}$ the indiscrete site on $\mC$, in other words, the only non-empty cover allowed in the topology is identity. 
In this article, a Tate algebra over a $p$-adic field $K$ refers to classical Tate algebras, which are quotients of $K\gr{t_1, ..., t_d}$. A Huber ring with a topological nilpotent unit is called a Tate Huber ring. Finally, for an $\mO_C$-algebra $S$, we write $\Omega_{(S/p)}$ (resp. $\L \Omega_{(S/p)}$) to denote $\Omega_{(S/p)/(\mO_C/p)}$ (resp.  $\L\Omega_{(S/p)/(\mO_C/p)}$) to ease notations. 

\subsection*{Acknowledgements} 
The author is grateful to Bhargav Bhatt for numerous helpful correspondences and discussions on this subject. His comments and suggestions improved various parts of the original writeup. In addition, he would like to thank Lin Chen, Kiran Kedlaya, Mark Kisin, Shizhang Li, Akhil Mathew, Matthew Morrow and Sasha Petrov for helpful discussions in preparing this article.


\section{$A \Omega_S$ for quasisyntomic rings} \label{sec:QRSP}
 
\subsection{The quasisyntomic site}  

Recall from \cite{THH} the following definitions. 
\begin{definition} 
A  ring $A$ is quasisyntomic if  it is $p$-complete with bounded $p^\infty$-torsion and the base change of the cotangent complex $\L_{A/\Z_p} \otimes_A^\L A/p\in D(A/p)$ has Tor amplitude in $[-1, 0]$.  
\end{definition} 
Examples of quasisyntomic rings include $p$-completions of smooth algebras over perfectoid rings, 
and $p$-complete noetherian local complete intersections. In fact, by a result of Avramov, the latter essentially gives all noetherian quasisyntomic rings (see Theorem 4.13 in \textit{loc.cit}.).  

\begin{definition} 
Let $A, B$ be $p$-complete rings with bounded $p^\infty$-torsion. A morphism $f: A \ra B$ is quasisyntomic if
\bi[--]
\item $B \otimes_A^\L A/p \in D(A/p)$ lives in degree $0$ and is flat over $A/p$; 
\item $\L_{B/A} \otimes_B^\L B/p \in D(B/p)$ has Tor amplitude $[-1, 0]$.  
\ei
The morphism $f$ is proj(= projective)-quasisyntomic if 
\bi[--]
\item both $A$ and $B$ are $p$-torsion free;
\item $B/p$ is projective over $A/p$; 
\item $\L_{(B/p)/(A/p)} \in D(B/p)$ has projective amplitude in $[-1, 0]$.  
\ei 
\end{definition} 

We denote by $\qSyn_{\mO_C}^{\textup{op}}$ the site whose underlying category is the opposite category of quasisyntomic $\mO_C$-algebras, with topology generated by quasisyntomic covers. Let $\qSyn_{\mO_C}^{\textup{proj}} \subset \qSyn_{\mO_C}$ be the full subcategory spanned by proj-quasisyntomic $\mO_C$-algebras, and $\qSyn_{\mO_C}^{\textup{proj,op}}$ be the site whose topology is generated by proj-quasisyntomic covers.  By Lemma 4.16 (and Variant 4.35) in \textit{loc.cit.}, $\qSyn_{\mO_C}^{\textup{op}}$ and $\qSyn_{\mO_C}^{\textup{proj,op}}$ indeed form sites, in particular, the property of being quasisyntomic is preserved under quasisyntomic maps.  

\subsection{Quasiregular semiperfectoid rings} 
Now we look at a particularly useful basis for the topology of $\qSyn_{\mO_C}^{\textup{op}}$ (resp. $\qSyn_{\mO_C}^{\textup{proj,op}}$). 
\begin{definition} A ring $S$ is quasiregular semiperfectoid if it is quasisyntomic and there exists a surjective map $R_{\textup{perfd}} \twoheadrightarrow S$ from a perfectoid ring $R_{\textup{perfd}}$.
\end{definition}  The latter condition implies that $S$ is semiperfect (i.e. the Frobenius on $S/pS$ is surjective). Denote by $\qrsP_{\mO_C} \subset \qSyn_{\mO_C}$ the full subcategory of quasiregular semiperfectoid $\mO_C$-algebras, and write $\qrsP_{\mO_C}^{\textup{proj}}:= \qSyn_{\mO_C}^{\textup{proj}}\cap \qrsP_{\mO_C}.$   The category $\qrsP^{\textup{proj,op}}_{\mO_C}$ again forms a site under the topology given by proj-quasisyntomic covers. We will frequently use the fact that quasiregular semiperfectoid rings form a basis for the quasisyntomic site. More precisely,

\begin{lemma}[Proposition 4.30, Variant 4.35 in \cite{THH}] The natural restriction of sites of topoi
$\qrsP_{\mO_C}^{\textup{proj,op}} \ra \qSyn_{\mO_C}^{\textup{proj,op}}$ induces an equivalence 
$$ \text{Shv} ( \qSyn_{\mO_C}^{\textup{proj,op}}) \isom \text{Shv}( \qrsP_{\mO_C}^{\textup{proj,op}}).$$
\end{lemma}

\subsection{Flat/quasisyntomic descent} 

Recall from the introduction the functor from the category of $p$-adic completions of smooth $\mO_C$-algebras to derived $(p, \xi)$-complete $\Ainf$-algebras, sending
$$R \longmapsto A \Omega_R : = R \Gamma (\spf R,  L\eta_\mu R \nu_* \A_{\textup{inf}}).$$
By left Kan extension in $(p, \xi)$-complete $\Ainf$-complexes, we obtain a functor $A \Omega: R \mapsto A \Omega_R$ from all $p$-complete (simplicial) algebras to derived $(p, \xi)$-complete $\Ainf$-algebras, which satisfies the following de Rham comparison 
$$A \Omega_R \otimes^\L \Ainf/\xi \cong \L \Omega_{R/\mO_C}.$$ To relate to the crystalline cohomology, we base change to $\Acris$ to obtain a functor 
$$A \Omega \widehat\otimes^\L \Acris: R \longmapsto A \Omega_R \widehat\otimes^\L \Acris.$$
Here the completion is $(p, \mu)$-adic.\footnote{Note that the $p$-adic topology on $\Acris$ agrees with the $(p, \mu)$-adic topology since $\mu^{p-1}$ is divisible by $p$.}
In order to compute $R\Gamma_{\Ainf}(\fX) \widehat \otimes^\L \Acris$, we need the following lemma on flat (more precisely, quasisyntomic) descent.  

\begin{lemma} \label{lemma:tensor_with_Acris_is_sheaf} The functor 
$R \mapsto A \Omega_R \widehat \otimes_{\Ainf}^\L \Acris$ forms a sheaf on $   \qSyn_{\mO_C}^{\textup{proj, op}}$. 
\end{lemma}

\begin{proof} 
It suffices to show that $R \mapsto A \Omega_R \otimes_{\Ainf}^\L \Acris/p$ is a sheaf. 
Now observe that $\Acris$ is isomorphic to the $p$-adic completion of $\Ainf[\frac{\xi^{p^k}}{p^{1+p+\cdots + p^{k-1}}}]_{k \ge 1}$, therefore, 
$$\Acris/p \cong \Ainf[X_{1}, X_{2}, ...]/(p, \xi^p, X_1^p, X_2^p, ...)$$ 
is a countable direct sum of copies of $\Ainf/(p, \xi^p)$ as an $\Ainf$-module. Thus it suffices to show that $R \mapsto A \Omega_R \otimes^\L \Ainf/(p, \xi^p)$ is a coconnective sheaf. This follows from Construction 9.5 of \cite{THH}, which shows that $R \mapsto A \Omega_R$ is a coconnective sheaf on $\qSyn_{\mO_C}^{\textup{proj, op}}$ (the key input is that flat descent holds for cotangent complexes, see Theorem 3.1 of \textit{loc.cit}). 
\end{proof}

\subsection{$A \Omega_S$ for quasiregular semiperfectoid rings} 
Locally on quasiregular semiperfectoid rings $S$,  $A \Omega_S$ is a relatively simple object. 
\begin{lemma}
 For $S \in \qrsP_{\mO_C}^{\textup{proj}}$, $A \Omega_S \widehat \otimes^\L \Acris$ is a topologically free $\Acris$-module concentrated in degree $0$. 
\end{lemma}
 
\begin{proof} First note that $\L \Omega_{(S/p)}$ is a free $\mO_C/p$-module concentrated in degree $0$ (by considering the graded pieces $\wedge^i \L_{(S/p)} [-i]$ for the conjugate filtration on $\L \Omega_{(S/p)}$).  The $D(\Ainf)$-valued sheaf $R \mapsto A \Omega_R$ on $\qSyn_{\mO_C}^{\textup{proj, op}}$ takes discrete values on quasi-regular semiperfectoid objects, hence the lifting of a basis from $A \Omega_S \otimes^\L \Ainf/(p, \xi) \cong \L \Omega_{(S/p)}$ 
to $A \Omega_S$ becomes a topological basis, 
so $A \Omega_S$ is topologically free. 
To see that  $A \Omega_S \widehat \otimes^\L \Acris$ is topologically free over $\Acris$, we write $\Acris/p \cong \big(\Ainf/(p, \xi^p) \big)^{\oplus J}$ as in the proof of Lemma \ref{lemma:tensor_with_Acris_is_sheaf}, and observe that
$$A\Omega_S \otimes^\L \Acris/p \cong \big(A \Omega_S \otimes^\L \Ainf/(p, \xi^p) \big)^{\oplus J},$$
which is a discretely-valued free $\Acris/p \cong  \big(\Ainf/(p, \xi^p) \big)^{\oplus J}$ module. 
\end{proof}

\begin{lemma} \label{lemma:PD_thickening}  Retain notations from above. The natural projection map $\beta$
\[
\begin{tikzcd} 
A \Omega_S \widehat \otimes^\L \Acris \arrow[r] \arrow[rd, dashed, "\beta"] &  \L \Omega_{S/\mO_C} \arrow[d, two heads] \\
& S/p
\end{tikzcd} 
\] 
 is a PD-thickening. Here the horizontal arrow is induced from the de Rham comparison $A \Omega_S  \otimes^\L \mO_C \cong \L \Omega_{S/\mO_C}$, while the vertical map is the composition of $\L \Omega_{S/\mO_C} \ra \L \Omega_{(S/p)}$ with the augmentation map  $\L \Omega_{(S/p)} \ra \Omega^{\bullet}_{(S/p)} = S/p.$ 
\end{lemma} 

\begin{proof} 
We first observe that the statement is equivalent to that $A \Omega_S \widehat \otimes^\L \Acris \ra S$ is a PD-thickening, in other words, for any $x \in \sq I := \ker(A\Omega_S \widehat \otimes^\L \Acris \ra S)$, $x^n \in (n!) \cdot \sq I$. Moreover, since $S$ is $p$-torsion free, it suffices to show that, for each $x \in I: = \ker (\theta: A \Omega_S \ra S)$,  $x^p$ is divisible by $p$ in $A \Omega_S \widehat \otimes^\L \Acris$.\footnote{Note that the $p$-torsion-free condition on the target $S$ is necessary. For example, for a discrete valuation ring $R$ with uniformizer $\pi$ such that $\pi^p = p$, $R \ra R/\pi$ is not a PD-thickening (the requirement for the map $R \ra R/\pi$ to be a PD-thickening is precisely that the ramification index satisfies $e \le p - 1$).} 
Now consider the following commutative diagram 
\[
\begin{tikzcd}[column sep = 1.5em, row sep = 1.5em]
& A \Omega_S  \arrow[r, "\varphi"] \arrow[d] \arrow[ldd, swap, "\theta"] & 
A \Omega_S  \arrow[d]  \\ 
& A \Omega_S /(p, \xi) \arrow[r, "\varphi"] \arrow[d, equal] &  
A \Omega_S  /(p, \xi^p) \arrow[r, "\textup{can}"] & A \Omega_S /(p, \xi) \arrow[d, equal]\\
S \arrow[rd] & \L\Omega_{(S/p)} \arrow[d, "\alpha"] \arrow[rr, "\omega \: \longmapsto \:\omega^p"]  & &  \L \Omega_{(S/p)}  \\
& S/p
\end{tikzcd}
\]
where we denote by $\theta$ for the map $A \Omega_S \ra S$. Let $x$ be an element in $I = \ker (\theta)$, we wish to show that $x^p \in p \cdot A\Omega_S \widehat \otimes^\L \Acris$. Write $\cl x$ for the image of $x$ in $A \Omega_S/(p, \xi) \cong \L \Omega_{(S/p)}$. Now observe that, since $\alpha(\cl x) = 0 \in S/p$, we in fact have $(\cl x)^p = 0 \in \L\Omega_{(S/p)}$, 
hence $\varphi(\cl x) \in A \Omega_S /(p, \xi^p)$ lies in the kernel of the canonical projection denoted by $\textup{can}$ in the diagram. As $\ker (\textup{can})$ is generated by $\xi$, this implies that $\cl x \in (\varphi^{-1} (\xi)) \subset A \Omega_{S}/(p, \xi)$. In other words, we may write 
$$x = \varphi^{-1}(\xi) y + pz' + \xi w $$ 
for some $y, z', w \in A \Omega_S$. By assumption, $\theta(x) = u p^{1/p} \theta(y) + p \theta(z') = 0$ where $u$ is a unit in $S$, hence we know that $\theta(y) \in p^{(p-1)/p} S$. In other words, we may write $y = p^{(p-1)/p} y'+x_1$ with $y' \in A \Omega_S$ and $x_1 \in \ker (\theta)$. Inserting this expression into the previous expression for $x$, we have 
$$x = \varphi^{-1}(\xi) x_1 + p^{1/p} z + \xi w$$
where $z= p^{(p-2)/p} \varphi^{-1}(\xi) y' + p^{(p-1)/p} z'$. Now we may repeat this procedure with $x$ replaced by $x_1$. Note that $\big(\varphi^{-1}(\xi)\big)^p \equiv \xi \mod p$, therefore we conclude that $x \in (p^{1/p}, \xi)$, so $x^p$ is indeed divisible by $p$ in $A \Omega_S \widehat \otimes^\L \Acris$. 
\end{proof} 

\begin{remark} Another way to see the lemma 
is the following. Pick a surjection $\sq S \ra S$ from a perfectoid ring $\sq S$ with kernel $J$. Let $\sq S' = \sq S \gr{X_{j}^{1/p^\infty}}_{j \in J}$ be the perfectoid ring obtained by adjoining all $p$-power roots of $X_j$. 
Then consider the quasiregular semiperfectoid $\mO_C$-algebra $S' :=  \sq S \gr{X_{j}^{1/p^\infty}}_{j \in J}/(X_j)$ and choose a homomorphism $\sq S' \ra \sq S$ with $X_j \mapsto j$ for all $ j\in J$. This induces a surjective map $A \Omega_{S'} \widehat \otimes^\L \Acris \longrightarrow A \Omega_{S} \widehat \otimes^\L \Acris,$ 
and it suffices to prove the claim replacing $S$ with $S' =  \sq S \gr{X_{j}^{1/p^\infty}}_{j \in J}/(X_j)$. 
To this end, we let $S_0 =  \mO_C \gr{X_{j}^{1/p^\infty}}_{j \in J}/(X_j)$ and observe that the natural map 
$ A \Omega_{S_0} \widehat \otimes^\L_{\Ainf} \Ainf(\sq S) \ra A\Omega_{S'}$
is an isomorphism. 
This further reduces to proving that the kernel of $A \Omega_{S_0} \widehat \otimes^\L \Acris \ra S_0$ carries a PD-structure, where $S_0 = \mO_C \gr{X_{j}^{1/p^\infty}}_{j \in J}/(X_j)$, 
which should then follow from an explicit computation. 
\end{remark} 

\section{The crystalline nature of the $\Acris$ base change of $\Ainf$-cohomology}

As in the introduction, let $\fX$ be a smooth formal scheme over $\mO_C$.  In this section we construct a functorial map from $R\Gamma_{\textup{crys}} (\fX_{\mO_C/p}/\Acris) \ra R \Gamma_{\Ainf}(\fX) \widehat \otimes^\L \Acris$ which reduces to the de Rham comparison over $\mO_C/p$. It is convenient to view the $\Ainf$-cohomology as a $D(\Ainf)$-valued sheaf on $\fX_{\ett}$, given by 
$$\fU \longmapsto R \Gamma_{\Ainf} (\fU, A \Omega_{\fX}),$$
where $A\Omega_{\fX} = L \eta_{\mu} R \nu_* \Ainfx$. 
The functor  
$\fU \mapsto R \Gamma_{\Ainf} (\fU) \widehat \otimes^\L \Acris$
 forms a sheaf on $\fX_{\ett}$ by the same argument for Lemma \ref{lemma:tensor_with_Acris_is_sheaf} (here the completion is $p$-adic). Furthermore, as affine opens $\fU = \spf R$ in $\fX_{\ett, \textup{aff}}$ form a basis for $\fX_{\ett}$,  it suffices to construct a functorial map 
$$ R\Gamma_{\textup{crys}} ((R/p)/\Acris) \longrightarrow A\Omega_R \widehat \otimes^\L \Acris $$ 
for all $\spf R = \fU \subset \fX_{\ett, \textup{aff}}$. 

\subsection{The local map $h_{\textup{crys}}$}
\begin{construction} \label{con:functorial_map_h_crys}
Let $R$ be the $p$-adic completion of a smooth $\mO_C$-algebra. 
For each $\mO_C$-algebra map $R \ra S$ where $S\in \qrsP_{\mO_C}^{\textup{proj}}$, we have the following commutative diagram 
\[ 
\begin{tikzcd}[column sep = 1.5em]
\Acris \arrow[d] \arrow[rr]  & & A\Omega_S \widehat \otimes^\L \Acris \arrow[d, "\beta"] \\
\mO_C/p \arrow[r] & R/p \arrow[r] & S/p
\end{tikzcd}.
\]
The map $\beta$ is $A \Omega_S \widehat \otimes^\L\Acris \ra A\Omega_S/\xi \xrightarrow{\gamma_{\textup{dR}}} \L \Omega_{(S/p)} \ra S/p$ as described in Lemma \ref{lemma:PD_thickening}. By Lemma \ref{lemma:PD_thickening}, the diagram (often denoted by $ A\Omega_S \widehat \otimes^\L \Acris $) is an object in the big crystalline site of $R/p$ over $(\Acris, \mO_C/p, \gamma)$. Thus by restriction in the crystalline site we obtain a functorial map 
$$h_S: R\Gamma_{\textup{crys}} ((R/p)/\Acris) \longrightarrow  A\Omega_S \widehat \otimes^\L \Acris.$$
\end{construction} 

We first claim that 
\begin{lemma} \label{lemma:compatible_with_dR_bc} $h_S$ is compatible with the de Rham comparison along the base change $\Acris \ra \mO_C/p$. More precisely, the following diagram commutes
\[
\begin{tikzcd} 
R \Gamma_{\textup{crys}} ((R/p)/\Acris) \arrow[r, "h_{\textup{crys},S}"] \arrow[d] & A \Omega_S \widehat \otimes^\L \Acris  \arrow[d] \\
R\Gamma_{\textup{dR}} (R/p) \cong \Omega^{\bullet}_{(R/p)}  \arrow[r] &  \L \Omega_{(S/p)} 
\end{tikzcd}
\]
where the bottom map is the functorial map $\Omega^{\bullet}_{(R/p)} \cong \L \Omega_{(R/p)} \ra \L \Omega_{(S/p)}$, and the right vertical map is the composition of base change and the de Rham comparison $\gamma_{\textup{dR}}:  A \Omega_{S}/(p, \xi) \isom \L \Omega_{(S/p)}$.  
\end{lemma}

\begin{proof} 
By functoriality of the crystalline site, we may replace the top horizontal map by $\cl h_{\textup{crys},S}: R \Gamma_{\textup{crys}}\big((R/p)/(\mO_C/p)\big) \ra \L \Omega_{(S/p)}$. As the crystalline cohomology of $R/p$ over the trivial PD-ring $(\mO_C/p, 0)$ computes the de Rham cohomology, the problem is now the following: we have two maps $\cl h_{\textup{crys}}$ and $\cl h_{\textup{can}}$ from $R\Gamma_{\textup{dR}}(R/p)$ to $\L \Omega_{(S/p)}$, one constructed by viewing $\L\Omega_{(S/p)}$ as an object in the crystalline site, the other induced by functoriality of derived de Rham cohomology, and we need to show that the two maps agree (up to homotopy). 
For this we first reduce to the case where $R/p$ is a polynomial ring over $\mO_C/p$, by considering a surjective map $\Sigma \cong \mO_C/p [X_i] \twoheadrightarrow R/p$ (and view $\L \Omega_{(S/p)}$ as an object in the crystalline site over $\Sigma$ via the map $\Sigma \ra R/p$).  

In the rest of the proof we assume that $R/p = \Sigma$ is a polynomial ring over $\mO_C/p$. Now let $\sq S \twoheadrightarrow S$ be a perfectoid $\mO_C$-algebra which surjects onto $S$ (as $S$ is quasiregular semiperfectoid). Choose a homomorphism $\Sigma \ra \sq S/p$ such that its composition with the surjection $\sq S/p \twoheadrightarrow S/p$ gives the map $\Sigma \ra S/p$. By functoriality of derived de Rham cohomology we obtain a map $\L \Omega_{(\sq S/p)} \ra \L \Omega_{(S/p)}$, which is a morphism in the crystalline site over $\Sigma$. 
Thus the composition $$R \Gamma_{\textup{crys}}\big(\Sigma/(\mO_C/p)\big) \xrightarrow{\cl h_{\textup{crys}, \sq S}} \L \Omega_{(\sq S/p)} \longrightarrow \L \Omega_{(S/p)}$$
agrees with $\cl h_{\textup{crys}, S}$,
and we are further reduced to the case where $S$ is perfectoid, which we now assume. In this case $\L \Omega_{(S/p)} = S/p$, and by construction the map $\cl h_{\textup{crys}}$ is given by $\Sigma (\bullet) \longrightarrow S/p$ of cosimplicial algebras, where $\Sigma (n)$ is the $(n+1)$-folded product of the polynomial ring $\Sigma = R/p$ in the crystalline site (equivalently, the PD-envelop of $\Sigma^{\otimes(n+1)} \twoheadrightarrow \Sigma$),  and $S/p$ denotes the constant cosimplicial algebra. 
By the proof of the crystalline and de Rham  comparison in \cite{BdJ} (or similarly, by the proof of Proposition \ref{prop:compare_C_inf_coh_with_dR}), we see that the map $\cl h_{\textup{crys}}$ is given by 
$$\Sigma (\bullet) \cong \Omega^\bullet_{\Sigma/(\mO_C/p)} \longrightarrow S/p,$$
where the second map is $\Sigma \ra S/p$ in degree $0$, and $0$ in other degrees. This agrees with the description of $\cl h_{\textup{can}}$.
\end{proof}

\begin{remark} 
Here is another way to proceed once we reduce to the case of polynomial algebra. It suffices to compare $\cl h_{\textup{crys}}$ and $\cl h_{\textup{can}}$ on the graded pieces for the conjugate filtration. On the $\text{zero}^{th}$-graded piece the two maps agree. 
It suffices to check that they agree on the first graded piece (up to homotopy).  
To this end, observe that the map $\Omega^1_{(R/p)}[-1] \ra \L_{(S/p)}[-1]$ coming from $\cl h_{\textup{can}}$ is null-homotopic. For the map induced from $\cl h_{\textup{crys}}$, one can write down an explicit homotopy $H: \tau^{\le1} \Omega^{\bullet}_{(R/p)}/\ker(d^0) \ra \L \Omega_{(S/p)}/\L \tau^{\le0} \Omega_{(S/p)} [-1]$ between $\cl h_{\textup{crys}}$ and $0$, using the fact that $R/p$ is a polynomial algebra over $\mO_C/p$. 
\end{remark} 

\subsection{The global map $h_{\textup{crys}}$}
\begin{construction} \label{con:functorial_map_h_crys_2}
Since the functor $R \mapsto A \Omega_R \widehat \otimes^\L \Acris$ forms a sheaf on the quasisyntomic site (Lemma \ref{lemma:tensor_with_Acris_is_sheaf}), by taking the homotopy limit over maps  $R \ra S$ in the quasisyntomic site with $S \in  \qrsP_{\mO_C}^{\textup{proj}}$, we arrive at the desired homomorphism
\[
\begin{tikzcd}[
  column sep=1.5em,
  ar symbol/.style = {draw=none,"\textstyle#1" description,sloped},
  isomorphic/.style = {ar symbol={\cong}},
  ]
R\Gamma_{\textup{crys}} ((R/p)/\Acris) \arrow[rr, "h_{\textup{crys}, R}"]  
&&    \textup{R}\!\lim A\Omega_S \widehat \otimes^\L \Acris   \ar[r,isomorphic] 
& A\Omega_R \widehat \otimes^\L \Acris. 
\end{tikzcd}
\] 
As discussed in the beginning of the section, we then take the homotopy limit over formal affine opens $\spf R$ in the $\etale$ site  $\fX_{\ett}$, which allows us to write $R \Gamma_{\Ainf}(\fX) \widehat \otimes^\L \Acris \cong \textup{R}\!\lim A\Omega_R \widehat \otimes^\L \Acris$. Consequently, we obtain the construction of $h_{\textup{crys}}$ as the limit of $h_{\textup{crys}, R}$.  

\begin{remark} \label{remark:alternative_construction_using_THH}
An equivalent construction of the map $h_S$ can be given as follows. As in \cite{THH}, let $\A_{\textup{cris}}(S/p)$ be the $p$-completed PD-envelop of $\Ainf(S) \ra S/p$, and let $\L W\Omega_{S/p}$ be the derived de Rham--Witt complex of $S/p$ over $\F_p$, which is defined via left Kan extension from smooth $\F_p$-algebras. In particular $\L W \Omega_{S/p}$ is equivalent to the derived crystalline cohomology $\L R\Gamma_{\textup{crys}} (S/p)$ (over the PD-ring $(\Z_p, p)$). By Proposition 8.12 of \cite{THH}, there is a canonical isomorphism $\A_{\textup{cris}}(S/p) \isom \L R\Gamma_{\textup{crys}} (S/p)$. Thus there is a natural map 
$ \psi_S:\A_{\textup{cris}}(S/p) \longrightarrow A \Omega_S \widehat \otimes^\L \Acris$ by Lemma \ref{lemma:PD_thickening}. 
This in turn induces a map 
$$ R\Gamma_{\textup{crys}} ((R/p)/\Acris)  \longrightarrow   \L R\Gamma_{\textup{crys}} (S/p) \isom \A_{\textup{cris}}(S/p)  \longrightarrow  A \Omega_S \widehat \otimes^\L \Acris. $$
It might be possible to show that $\psi_S$ (hence $h_{\textup{crys}}$) is an isomorphism, which we do not pursue as the existence of $h_{\crys}$ already suffices to deduce the main results of \cite{BMS} (essentially all except for $\Acris$-specialization). Moreover, we prefer Construction \ref{con:functorial_map_h_crys_2} for the following reasons. First, this avoids using Proposition 8.12 in  \cite{THH}.\footnote{which is elementary enough but still involves analyzing the conjugate filtration on a certain PD envelop, see Proposition 8.11 thereof.} More importantly, Construction \ref{con:functorial_map_h_crys_2} (via the crystalline site) compares well with the $\Bdrp$-cohomology by functoriality. Finally, our construction generalizes to the logarithmic setting, while $\psi_S$ does not take log structures into account. 
\end{remark} 
\end{construction}

We summarize this discussion in the following 

\begin{theorem} \label{thm:construction_of_h_cris}
Let $\fX$ be a smooth formal scheme over $\mO_C$. There exists an $\Acris$-linear map 
$$h_{\textup{crys}} = h_{\textup{crys}, \fX}: R \Gamma_{\textup{crys}} (\fX_{\mO_C/p}/\Acris) \longrightarrow R\Gamma_{\Ainf} (\fX) \widehat \otimes^\L_{\Ainf} \Acris $$ 
which is $\varphi$-equivariant and functorial on $\fX$. Moreover, it is compatible with the de Rham comparison. In other words, the following diagram commutes:
\[
\begin{tikzcd}
R\Gamma_{\textup{crys}} ((R/p)/\Acris) \arrow[r] \arrow[d] &  R\Gamma_{\Ainf} (\fX) \widehat \otimes^\L \Acris \arrow[d] \\
 R\Gamma_{\textup{dR}} (\fX_{\mO_C/p})  \arrow[r, " \gamma_{\textup{dR}}^{-1}"] & R\Gamma_{\Ainf}(\fX) \otimes^\L_{\Ainf, \theta} \mO_C/p
\end{tikzcd}
\] 
\end{theorem}

\subsection{The crystalline specialization of $\Ainf$-cohomology} \label{ss:crystalline_specialization}

As a direct corollary of the construction above (Theorem \ref{thm:construction_of_h_cris}), we obtain the  crystalline comparison of $\Ainf$-cohomology. 

\begin{corollary} \label{cor:crystalline_comparison} Let $\fX$ be a smooth formal scheme over $\mO_C$.  There is a canonical $\varphi$-compatible quasi-isomorphism 
$$h_{W}:  R\Gamma_{\textup{crys}} (\fX_{k}/W(k)) \isom R\Gamma_{\Ainf} (\fX) \widehat \otimes^\L_{\Ainf, \vartheta} W(k)$$
relating $\Ainf$-cohomology of $\fX$ to the crystalline cohomology of $\fX_k$. 
\end{corollary}

\begin{proof} 
We base change $h_{\crys}$ along $\vartheta: \Acris \ra W(k)$ to obtain the desired morphism. 
Now we further base change from $W(k) \ra k$ and consider the following obviously commutative diagram 
\[
\begin{tikzcd}[row sep = 1.2em, column sep = 1.5em]
\Acris \arrow[r, "\vartheta"]  \arrow[d, "\theta"] & W(k) \arrow[d] \\
\mO_C/p \arrow[r] & k
\end{tikzcd}
\]
of maps of rings. 
Therefore, the base change of $h_W$ to $k$ is $\gamma_{\textup{dR}}^{-1} \otimes^\L k$, which is a quasi-isomorphism, then apply derived Nakayama's lemma. 
\end{proof} 

\begin{remark} 
We do not know whether this directly implies that $h_{\textup{crys}}$ is an quasi-isomorphism. The issue is that the ideal $I= \ker(\Acris \ra \mO_C/p)$ is not finitely generated, so we do not know whether derived $I$-completion is a well-behaved notion (in particular we do not know whether the derived Nakayama's lemma still holds in this context). 
\end{remark} 
 

\section{The $\Bcris$ comparison  theorem} \label{sec:Bcris}

In this section we analyze the cohomology groups $H^*_{\Ainf} (\fX)$, and prove the remaining part of Theorem \ref{mainthm:BKF_and_C_cris} from the introduction. Throughout this section we further assume that $\fX$ is proper.  

\subsection{The individual cohomology groups}
 
\begin{definition} \label{def:BKF_module}
A Breuil--Kisin--Fargues module is a pair $(M, \varphi_M)$ where $M$ is a finitely presented $\Ainf$-module such that $M [\frac{1}{p}]$ is free over $\Ainf[\frac{1}{p}]$, and $\varphi_M$ is a $\varphi$-linear isomorphism 
$\varphi_M: M [\frac{1}{\xi}] \isom M [\frac{1}{\varphi(\xi)}].$
\end{definition} 

We will need the following classification result of Fargues for the proof of Theorem \ref{thm:BMS_on_recoverfromgeneric}.\footnote{As remarked in \cite{BMS}, only the easy direction of the equivalence (namely fully faithfulness) is needed.} 

\begin{theorem}[Fargues] \label{thm:Fargues} There is an equivalence of categories between finite free Breuil--Kisin--Fargues modules and the category of pairs $(T, \Xi)$, where $T$ is a finite free $\Z_p$-module and $\Xi$ is a $\Bdrp$-lattice in $T \otimes_{\Z_p} \Bdr$, given by the functor
$$(M, \varphi_M) \mapsto \Big( (M \otimes W(C^\flat))^{\varphi_M = 1}, M \otimes \Bdrp \Big).$$
\end{theorem}

\begin{theorem}\label{thm:valuation_in_BKF} Let $\fX$ be a smooth proper formal scheme over $\spf \mO_C$.
The cohomology group $H^i_{\Ainf} (\fX)$ takes value in Breuil--Kisin--Fargues modules, and vanishes for $i > 2 \dim_{/\mO_C} \fX$.
\end{theorem}

\begin{remark} \label{remark:valuation_in_BKF}
The claim on vanishing of cohomology is immediate from the de Rham comparison. For the (implicit) claim on Frobenius, observe that the Frobenius on $R\nu_*\Ainfx$ induces the desired map on the level of sheaves 
$$ L\eta_{\mu} R\nu_*\Ainfx \isom L \eta_{\varphi(\xi)} L\eta_{\mu} R\nu_*\Ainfx \longrightarrow L\eta_{\mu} R\nu_*\Ainfx$$
on the pro-$\etale$ site, which becomes an isomorphism after inverting $\varphi(\xi).$ The first isomorphism essentially follows from definition of $L\eta$ (note that $\mu$ is regular in $\Ainf$). The existence of the second arrow is not entirely formal: it uses in particular that $\mH^0(L \eta_{\mu} R \nu_* \Ainfx)$ has no $\varphi(\xi)$-torsion and properties of the $L\eta$ operator (see Lemma 5.8 of \cite{Morrow}). 

It remains to show that $H^i_{\Ainf} (\fX)$ is finitely presented and becomes free after inverting $p$. Our approach follows the method outlined in \cite{Morrow}, which is different from the proof given in\cite{BMS}, where they need to identify $R\Gamma_{crys} (\fX_{\mO_C/p}/\Acris)$ with $R \Gamma_{\Ainf}(\fX) \otimes^\L \Acris$ (at least after inverting $p$). Our proof uses a careful descending induction on the degree $i$ of cohomology, where the most difficult part is to show that $H^i_{\Ainf} (\fX) [\frac{1}{p}]$ is finite free. We will complete the proof in the next subsection. 
\end{remark}

\subsection{Valuation in Breuil--Kisin--Fargues modules} \label{ss:BKF} 
We continue to assume that  $\fX$ is smooth proper over $\mO_C$. 
\begin{lemma} \label{lemma:invalid} 
Via base change along $\Acris \ra \Bdrp$, we have a functorial isomorphism $h_{\crys} \otimes \Bdrp$:   $$R \Gamma_{\crys} (\fX_{\mO_C/p}/\Acris) \otimes_{\Acris}^\L \Bdrp \isom R \Gamma_{\Ainf} (\fX) \otimes_{\Ainf}^\L \Bdrp.$$
\end{lemma}

\begin{proof} 
As mentioned in the introduction, the $\Ainf$-cohomology $R\Gamma_{\Ainf}$ is derived $\xi$-adically complete 
thus it is a perfect complex in $D(\Ainf)$ by the de Rham comparison. On the other hand, the crystalline cohomology $R \Gamma_{\crys} (\fX_{\mO_C/p}/\Acris)$ is a perfect complex in $D(\Acris)$ (see, for example, Tag 07MY in \cite{SP}).  In particular, both sides of the map given above are derived $\xi$-adically complete (to see this, note that $M \in D(\Bdrp)$ is derived complete if and only if each $H^i(M)$ is derived $\xi$-adically complete, and then observe that finitely presented modules over $\Bdrp$ are derived $\xi$-complete). Therefore, to prove the lemma it suffices to show that $h_{\crys} \otimes \Bdrp$ becomes an isomorphism after reducing mod $\xi$, this in turn follows from Theorem \ref{thm:construction_of_h_cris} and the de Rham comparison. 
\end{proof} 

We also need the following variant of the Berthelot--Ogus comparison isomorphism, as already observed in \cite{BMS}. For this lemma we fix a section $k \ra \mO_C/p$. 

\begin{lemma}[BMS] \label{lemma:crys_BO}
There is a natural $\varphi$-equivariant isomorphism $$ H^i_{\crys}(\fX_{\mO_C/p}/\Acris)[\frac{1}{p}] \isom H^i_{\crys} (\fX_k/W(k)) \otimes_{W(k)} \Bcrisp.$$  In particular, $ H^i_{\crys}(\fX_{\mO_C/p}/\Acris)[\frac{1}{p}]$ is a finite free $\Bcrisp$-module. 
\end{lemma}

\begin{proof} 
The point is that the Frobenius $\varphi$ on the crystalline cohomology $H^i_{\crys} (\fX_{\mO_C/p}/\Acris)$ becomes an isomorphism after inverting $p$ (since on affine opens $U$, there is an isomorphism $U \cong U_k \times_{\spec k} \spec \mO_C/p$ by the smoothness assumption, then apply base change). Moreover, there exists a large enough $n$ such that $$\iota: \fX_{\mO_C/p^{1/p^n}} \cong \fX_k \times_{\spec k} \spec \mO_C/p^{1/p^n},$$ 
and any two such isomorphisms agree once we enlarge $n$. Now by repeatedly applying Frobenius we obtain the following diagram  
\[
\begin{tikzcd}[row sep = 2em]
H^i_{\crys}(\fX_{\mO_C/p^{1/p^n}}/ \textup{A}^{(\varphi^n)}_{\textup{cris}} ) \otimes_{\varphi^n} \Bcrisp \arrow{r}{\sim}[swap]{\varphi^n} & H^i_{\crys}(\fX_{\mO_C/p}/\Acris) [\frac{1}{p}] \\
H^i_{\crys}(\fX_k /W(k)) \otimes_{\varphi^n} \Bcrisp \arrow{u} \arrow{r}{\varphi^n \otimes \textup{id}}[swap]{\sim}  
&  H^i_{\crys}(\fX_k /W(k)) \otimes \Bcrisp \arrow[u, dashed]
\end{tikzcd}
\]
where $ \textup{A}^{(\varphi^n)}_{\textup{cris}}$ denotes the $p$-completed PD-envelop of $\Ainf \twoheadrightarrow \mO_C/p^{1/p^n}$. The top horizontal isomorphism is induced by the absolute Frobenius $\varphi^n$ on $\fX_{\mO_C/p}$, which factors as $\fX_{\mO_C/p} \xrightarrow{\textup{pr}} \fX_{\mO_C/p^{1/p^n}} \xrightarrow{\varphi^n} \fX_{\mO_C/p}$, and the left vertical isomorphism comes from the identification $\iota$ noted above. 
\end{proof} 

\begin{corollary} \label{cor:same_rank}
Continue to assume that $\fX$ is smooth proper over $\mO_C$, then  $$\textup{rk}_{W(k)}  H^i_{\crys} (\fX_{k}/W(k))  = \textup{rk}_{\Z_p} H^i_{\ett} (X, \Z_p).$$
\end{corollary}

\begin{proof} 
By Lemma \ref{lemma:invalid} and the $\etale$ comparison for $\Ainf$-cohomology, we have an isomorphism  $R \Gamma_{\crys} (\fX_{\mO_C/p}/\Acris) \otimes_{\Acris}^\L \Bdr \cong R \Gamma_{\ett} (X, \Z_p) \otimes \Bdr.$ Denote their common value 
by $D$, and for ease of notation write $H^i_{\crys} (\fX_{\mO_C/p})$ for the crystalline cohomology of $\fX_{\mO_C/p}$ over $\Acris$, then $$ H^i_{\crys} (\fX_{\mO_C/p}) \otimes_{\Acris} \Bdr \cong H^i(D)$$ by the spectral sequence $\textup{Tor}_a (H^{b}_{\crys} (\fX_{\mO_C/p}), \Bdr) \so H^{b-a}(D)$.
Note that all higher ($a \ge 1$) Tor terms $\textup{Tor}_a^{\Acris} (H^{b}_{\crys} (\fX_{\mO_C/p}), \Bdr)$ vanishes, since $p$ is already inverted, and $H^i_{\crys}(\fX_{\mO_C/p}/\Acris)[\frac{1}{p}]$ is finite free by Lemma \ref{lemma:crys_BO}. 
Similarly, we have the following isomorphism $$H^i_{\crys} (\fX_{\mO_C/p}) \otimes_{\Acris} \Bdr \cong H^i_{\ett} (X, \Z_p) \otimes_{\Z_p} \Bdr$$ of free $\Bdr$-modules.  The lemma then follows from Lemma \ref{lemma:crys_BO} once again. 
\end{proof} 

\begin{lemma} \label{lemma:relating_Tor_and_crystalline} 
\textit{Assume} that $H^{k}_{\Ainf} (\fX)[\frac{1}{p}]$ is finite free $\Ainf[\frac{1}{p}]$-modules for $k \ge i+1$, then the cohomology groups fit in the following short exact sequences  $$ 0 \ra H^i_{\Ainf} (\fX) \otimes W(k) \ra H^i_{\textup{crys}} (\fX_k/W(k)) \ra \textup{Tor}^1(H^{i+1}_{\Ainf}(\fX), W(k)) \ra 0.$$  In particular, if $H^i_{\textup{crys}} (\fX_k/W(k))$ is torsion free, then so is $ H^i_{\Ainf} (\fX) \otimes W(k)$. 
\end{lemma} 

\begin{remark} This lemma is used in the proof of Theorem \ref{thm:valuation_in_BKF}. Of course, once that theorem is proven, the assumption here is no longer necessary. 
\end{remark}
\begin{proof} 
By the crystalline comparison (the derived version, see Corollary \ref{cor:crystalline_comparison}), the spectral sequence for Tor becomes $E_2^{-a, b} = \textup{Tor}^{\Ainf}_a(H^b_{\Ainf}(\fX), W(k)) \so H^{b-a}_{\crys} (\fX_k/W(k)).$ Now apply Lemma A.5 in \cite{Morrow}, which implies that 
$ \textup{Tor}^{\Ainf}_a(H^b_{\Ainf}(\fX), W(k)) = 0$ for all $b > i, a > 1$. 
\end{proof} 

The final ingredient we need is the following lemma of Morrow (which replaces the role of Lemma 4.19 of \cite{BMS}). 

\begin{lemma}[Morrow] \label{lemma:Morrow_criterion_for_free}
 Let $M$ be a finitely presented $\Ainf$-module equipped with a $\varphi$-semilinear endomorphism, which becomes an isomorphism upon inverting $\xi$. Assume that $M[\frac{1}{p\mu}]$ is a finite free $\Ainf [\frac{1}{p \mu}]$-module of the same rank of $M \otimes_{\Ainf} W(k)$. Then $M[\frac{1}{p}]$ is finite free over $\Ainf[\frac{1}{p}]$. 
\end{lemma}

\begin{proof} 
This is Lemma A.4 of \cite{Morrow}. 
\end{proof} 
 
Now we are ready to prove Theorem \ref{thm:valuation_in_BKF} (using the method of \cite{Morrow}). 

\begin{proof}[Proof of Theorem \ref{thm:valuation_in_BKF}]
As in Remark \ref{remark:valuation_in_BKF} it remains to show that $H^i_{\Ainf} (\fX)$ is finitely presented and becomes free after inverting $p$, which we prove by a descending induction on $i$.  The claim is vacuously true for $i > 2  \dim \fX$.  By inductive hypothesis, each $H^k(\tau^{> i} R \Gamma_{\Ainf}(\fX) )$ is finitely presented over $\Ainf$ and becomes free after inverting $p$, so as a complex it is perfect over $\Ainf$, so the bounded complex $\tau^{> i} R \Gamma_{\Ainf}(\fX)$ is perfect, and hence $\tau^{\le i} R \Gamma_{\Ainf}(\fX)$ is perfect. From this it follows that the cohomology $H^i_{\fX}(\fX)$ is finitely presented as it is the cokernel of a map between two finite projective $\Ainf$-modules. To show that $H^i_{\Ainf} (\fX) [\frac{1}{p}]$ is finite free over $\Ainf[\frac{1}{p}]$, we apply Lemma \ref{lemma:Morrow_criterion_for_free}. Note that from the $\etale$ comparison (and note that $\Ainf[\frac{1}{\mu}]$ is $p$-torsion free) we have $H^i_{\Ainf}(\fX) [\frac{1}{\mu}] = H^i_{\ett}(X, \Z_p) \otimes_{\Z_p} \Ainf[\frac{1}{\mu}],$ thus it suffices to show that $\Z_p$-rank of $H^i_{\ett}(X, \Z_p)$ is the same the $W(k)$-rank of $H^i_{\crys} (\fX_k/W(k))$. This is precisely the statement of Corollary \ref{cor:same_rank}. 
\end{proof}  

\begin{remark} \label{remark:torsion}
As mentioned in the introduction, Theorem \ref{thm:valuation_in_BKF} together with the comparison theorems of $R\Gamma_{\Ainf}(\fX)$ is enough to deduce the following results of \cite{BMS} on torsion. First of all, for each $n \ge 1$, we have  $$\textup{length}_{W(k)} \big(H^i_{\textup{crys}}(\fX_k/W(k))_{\textup{tor}} /p^n \big) \ge \textup{length}_{\Z_p} \big(H^i_{\ett}(X, \Z_p)_{\textup{tor}} /p^n \big). $$ Moreover, $H^i_{\textup{crys}}(\fX_k/W(k))$ is $p$-torsion free if and only if $H^i_{\textup{dR}} (\fX/\mO_C)$ is $p$-torsion free, in which case, $H^i_{\Ainf}(\fX)$ is a finite free $\Ainf$-module. These claims follow from some basic properties of finitely presented $\Ainf$-modules, notably Lemma 4.15, 4.17 and 4.18 of \textit{loc.cit}, for which it is important to know that $H^i_{\Ainf}(\fX)[\frac{1}{p}]$ is free.  
\end{remark}  

\begin{remark} \label{remark:h_crys_on_cohomology}
Note that after inverting $p$, the map $h_{\crys}$ induces the following $\varphi$-equivariant map on cohomology groups (which we denote again by $h_{\crys}$) $$h_{\crys}:  H^i_{\textup{crys}}(\fX_{\mO_C/p}/\Acris) \otimes_{\Acris} \Bcrisp \ra H_{\Ainf}^i(\fX) \otimes \Bcrisp,$$ which is functorial on $\fX$. This directly follows from the freeness of $H^i_{\Ainf}(\fX)[\frac{1}{p}]$. 
\end{remark}

\subsection{The $\Bcris$ comparison } \label{ss:Bcris_comparison}

In this subsection we prove the $\Bcris$ comparison theorem (except for the part on filtration, which we leave to the last section of this paper). We adopt the following setup: let $K$ be a finite extension of $\Q_p$,\footnote{or more generally a discretely valued nonarchimedean extension}  with residue field $k_0$, and maximal unramified subfield $K_0= W(k_0)[\frac{1}{p}]$. Let $C = C_K$ be a $p$-completed algebraic closure of $K$, with residue field $k \cong \cl \F_p$. Let $\fX_{\mO_K}$ be a smooth proper formal scheme over $\mO_K$, with adic generic fiber $X_K$ and special fiber $\fX_0$. Let $\fX$ be the base change of $\fX_{\mO_K}$ to $\mO_C$, so $X = X_K \times \spa(C, \mO_C)$ and $\fX_k = \fX_0 \times_{\spec k_0} \spec k$. Note that there is a unique section from $k = \cl \F_p \ra \mO_C/p$ in this setup. By combining Lemma \ref{lemma:crys_BO} and Remark \ref{remark:h_crys_on_cohomology}, we get a $\textup{Gal}_K$-equivariant (by functoriality) and $\varphi$-compatible map 
$$\alpha_{\crys}: H^i_{\crys} (\fX_0/W(k_0)) \otimes \Bcrisp \longrightarrow H_{\Ainf}^i(\fX) \otimes \Bcrisp.$$
More precisely $\alpha_{\crys}$ is the following composition
\[
\begin{tikzcd}[column sep = 2em]
H^i_{\crys} (\fX_0/W(k_0)) \otimes_{W(k_0)} \Bcrisp \arrow[d, "\sim"{sloped, above}] \arrow[r, dashed, "\alpha_{\crys}"] & H^i_{\Ainf} (\fX) \otimes \Bcrisp \\
H^i_{\crys} (\fX_k/W(k)) \otimes_{W(k)} \Bcrisp \arrow{r}{\sim}[swap]{\textup{by }\ref{lemma:crys_BO}} &  H^i_{\textup{crys}}(\fX_{\mO_C/p}/\Acris) \otimes \Bcrisp \arrow[u, "h_{\crys}"]
\end{tikzcd}
\]
By construction, the base change of $\alpha_{\crys}$ along $\Bcrisp \ra K_0$ becomes a $\varphi$-equivariant isomorphism $H^i_{\crys} (\fX_0/W(k_0))[\frac{1}{p}] \isom H_{\Ainf}^i(\fX) \otimes K_0$, which is precisely the map on cohomology groups obtained from the quasi-isomorphism $h_W$ in Corollary \ref{cor:crystalline_comparison}. 

\begin{theorem}Retain the setup from above. Then after composing with the $\etale$ comparison, $\alpha_{\crys}$ induces a functorial $(\textup{Gal}_K, \varphi)$-equivariant isomorphism 
$$ \beta_{\crys}:  H^i_{\crys} (\fX_0/W(k_0)) \otimes \Bcris \isom H^i_{\ett}(X, \Z_p) \otimes  \Bcris.$$ In particular, $H^i_{\ett} (X, \Q_p)$ is crystalline.  
\end{theorem}
\begin{proof} 
We consider the map $\alpha_{\crys}$ constructed above ands take the $\textup{Gal}_K$-invariant subspaces, which leads to the following diagram of $K_0$-vector spaces
\footnote{Note that all terms in this  commutative square are $K_0$-vector spaces, as $K_0 \subset (\Bcrisp)^{\textup{Gal}_K} \subset (\Bcris)^{\textup{Gal}_K} = K_0$.} 
\[ \begin{tikzcd}[column sep = 2em, row sep = 1.5em]
H^i_{\textup{crys}}(\fX_{0}/W(k_0)) [\frac{1}{p}] \arrow[d] 
\\
\Big(H^i_{\textup{crys}}(\fX_{0}/W(k_0)) \otimes \Bcrisp \Big)^{\textup{Gal}_K} \arrow[r]  
\arrow[d] 
& (H_{\Ainf}^i(\fX) \otimes \Bcrisp )^{\textup{Gal}_K} \arrow[d] \\
H^i_{\textup{crys}}(\fX_{0}/W(k_0)) [\frac{1}{p}]  \arrow{r}{\alpha_{\crys}\otimes K_0}[swap]{\sim} 
&  H^i_{\Ainf} (\fX) \otimes K_0 
\end{tikzcd} \]
Here the commutative square is obtained from applying $\textup{Gal}_K$-invariance to the morphism $\alpha_{\crys}$ and its base change $\alpha_{\crys} \otimes K_0$ (the Galois group acts trivially on the latter).  
The composition of the left vertical map (induced by $W(k_0) \ra \Acris \ra W(k_0)$) is an isomorphism, therefore the following composition 
$$ H^i_{\textup{crys}}(\fX_{0}/W(k_0)) [\frac{1}{p}]  \longrightarrow (H_{\Ainf}^i(\fX) \otimes \Bcrisp )^{G_K} \longrightarrow H^i_{\Ainf} (\fX) \otimes K_0 $$
is an isomorphism. In particular, we have 
$$\dim_{K_0} (H_{\Ainf}^i(\fX) \otimes \Bcrisp )^{\textup{Gal}_K} \ge \dim_{K_0} H^i_{\textup{crys}}(\fX_{0}/W(k_0)) [\frac{1}{p}].$$
Now let $V = H^i_{\ett} (X, \Q_p)$, then $ \textup{D}_{\textup{cris}}(V) \cong (H_{\Ainf}^i(\fX) \otimes \Bcris )^{\textup{Gal}_K} $ by the $\etale$ comparison, therefore we conclude that 
$$\dim_{K_0} \textup{D}_{\textup{cris}}(V) \ge \dim_{K_0} (H_{\Ainf}^i(\fX) \otimes \Bcrisp )^{\textup{Gal}_K} \ge \dim_{\Q_p} V $$
where the last inequality follows from the previous inequality and Corollary \ref{cor:same_rank}. The theorem thus follows.  
\end{proof} 

\begin{remark}  \label{remark:enough_Galois_inv}
Let us rephrase the proof in words. To show that $V$ is crystalline we want to show that $\textup{D}_{\textup{cris}}(V)$ is large enough (namely there are enough Galois invariants). We know that the crystalline cohomology $H^i_{\crys}(\fX_k/W(k))$ has large enough rank by Corollary \ref{cor:same_rank}, so we need to relate the crystalline cohomology to $\textup{D}_{\textup{cris}}(V)$ (which we know \textit{a posteriori} are the same after inverting $p$). This is precisely provided by $h_{\crys}$ (and $\alpha_{\crys}$ on the cohomology groups). It does not seem possible to deduce the $\Bcris$ comparison directly from Corollary \ref{cor:same_rank}.
\end{remark}


\section{The infinitesimal nature of $\Bdrp$-cohomology}

In this section we give a reformulation of the $\Bdrp$-cohomology of \cite{BMS}. We then use it to study the filtration compatibility of the $\Bcris$ comparison, and to recover $H^i_{\Ainf} (\fX)$ from the generic fiber in certain restricted cases. 

\subsection{An infinitesimal site} In this subsection let $X$ be a smooth proper rigid analytic variety over $\spa (C, \mO_C)$.
\footnote{We view $X$ as an adic space. In fact, it is possible to work the slightly more general class of analytic adic spaces that are locally noetherian -- the former condition in particular implies that $X$ can be covered by the adic spectrum of Tate Huber pairs.}  Let $\Bdrp_{,m}$ be the complete Tate Huber ring $\Bdrp_{,m} := \Bdrp/\xi^m  = \Ainf[\frac{1}{p}]/\xi^m$ over $\Q_p$.

\begin{definition} 
Let $A$ be a Tate algebra over $C$. The indiscrete infinitesimal site $\textup{Inf}(A/\Bdrp_{,m})^{\textup{op}, \circ}$ of $A$ relative to $\Bdrp_{,m} \twoheadrightarrow C$ is defined as follows. The objects of $\textup{Inf}(A/\Bdrp_{,m})^{\textup{op}}$ are of the form 
\[ 
(B, J) =  \begin{tikzcd}[column sep = 1.5em, row sep = 1em]
\Bdrp_{,m} \arrow[d] \arrow[rr]  & & B  \arrow[d] \\
C \arrow[r] & A \arrow[r] & B/J 
\end{tikzcd}
\]
where $B \ra B/J$ is pro-nilpotent thickening. Here $B$ is a complete Huber ring, whose topology is compatible with $\Bdrp_{,m}$ (so in particular it is Tate with $p$ being a pseudo-uniformizer). The topology of the site $\textup{Inf}(A/\Bdrp_{,m})^{\textup{op}, \circ}$ is the indiscrete topology.
\end{definition} 

Similarly we define the site $\textup{Inf}(A/\Bdrp)^{\textup{op}, \circ}$, with the requirement that $B = \varprojlim B/\xi^m$ (as a topological ring), and $(B/\xi^m, J) \in \textup{Ob}(\textup{Inf}(A/\Bdrp_{,m})).$ $\textup{Inf}(A/\Bdrp)^{\textup{op}, \circ}$ (resp. $\textup{Inf}(A/\Bdrp_{,m})^{\textup{op}, \circ}$) is a ringed site with the structure sheaf  $\mO_{\textup{Inf}}$, 
which sends $(B, J) \mapsto B$. 

\begin{definition} Let $K$ be any complete nonarchimedean extension of $\Q_p$. A smooth Tate algebra $A$ over $K$ is \textup{petit} if there exists an $\etale$ map $K\gr{t_i} \ra A$ from a Tate polynomial with finitely many variables.  
\end{definition} 

For a petit Tate algebra $A$ over $C$,  the infinitesimal cohomology $R \Gamma_{\textup{inf}}(A/\Bdrp)$ of $A$ with $\Bdrp$ coefficients is defined as the cohomology of the structure sheaf  $\mO_{\textup{Inf}}$.\footnote{Equivalently, we may take the derived pushforward of $\mO_{\textup{Inf}}$ along the map of topoi 
$$\textup{Shv}(\textup{Inf}(A/\Bdrp)^{\textup{op}, \circ}) \ra \textup{Shv} (\textup{Aff}(\Bdrp)^{\textup{op}, \circ}) $$ and then take the (derived) global sections there. Here the map is induced by the obvious cocontinuous map from $\textup{Inf}(A/\Bdrp)$ to the big affine ``Zariski'' site $\textup{Aff}(\Bdrp)^{\textup{op}, \circ}$ with indiscrete topology, where objects there are simply ring homomorphisms $\Bdrp \ra S$.} 
To define the global infinitesimal cohomology $R \Gamma_{\textup{inf}}(X/\Bdrp)$ of $X$, we let $X_{\min, \ett}$ be the site generated by petit affinoid objects in the $\etale$ site $X_{\ett}$. Note that the petit affinoids form a basis for $X_{\ett}$, by Corollary 1.6.10 of \cite{Huber}. We then define  
$$R \Gamma_{\textup{inf}}(X/\Bdrp) := \lim_{\spa(A, A^\circ)} R \Gamma_{\textup{inf}}(A/\Bdrp) $$
where the derived inverse limit is taken over all affinoids $\textup{sp} A = \spa (A, A^\circ)$ in $X_{\min, \ett}$. Equivalently, we may regard $R \Gamma_{\textup{inf}}(A/\Bdrp)$ as a functor from $X_{\min, \ett}$ to $D(\Bdrp)$, which takes values in derived $\xi$-complete objects.\footnote{This is justified in the proof of the lemma below.} This turns out to be a sheaf on $X_{\ett}$ (by the comparison with de Rham cohomology below), and thus induces a sheaf on $X_{\ett}$, whose derived global section gives the global infinitesimal cohomology. 

\begin{remark} 
It should be possible to define a global version of the infinitesimal site (with the $\etale$ topology) and then define $R \Gamma_{\textup{inf}}(X/\Bdrp)$ directly using this site. However, for simplicity we prefer to work locally on affinoids, for which it often suffices to consider the indiscrete topology by the vanishing of higher cohomology for coherent 
sheaves on affinoids (in our simplified setup this is already ensured by Tate's acyclicity theorem). 
\end{remark}

\subsection{Comparison with de Rham cohomology}

\begin{lemma} There is a quasi-isomorphism (obtained by base change) 
$$ R \Gamma_{\textup{inf}}(A/\Bdrp) \otimes^\L_{\Bdrp} C  \isom R \Gamma_{\textup{inf}}(A/C).$$
Note that $\textup{Inf}(A/\Bdrp_{,1})^{\textup{op}, \circ} = \textup{Inf}(A/C)^{\textup{op}, \circ}$. 
\end{lemma} 

\begin{proof} 
Choose a surjection $C\gr{X_i} \twoheadrightarrow A$ from a Tate polynomial ring to $A$. Let $\Bdrp \gr{X_i} = \varprojlim \Bdrp_{,m}\gr{X_i},$ where $\Bdrp_{,m} \gr{X_i} = \big(\Ainf/\xi^m \gr{X_i} \big) [\frac{1}{p}]$, and let $\Sigma_{\textup{dR}}$ be the classical
completion of  $\Bdrp\gr{X_i} $ with respect to the kernel 
$$J = J(0) := \ker(\Bdrp \gr{X_i} \twoheadrightarrow A).$$  $\Sigma_{\textup{dR}}$ is a weakly terminal object in $\textup{Inf}(A/\Bdrp)^{\textup{op}, \circ}$. Its $(n+1)$-folded product $\Sigma_{\textup{dR}}(n)$ is given by the completion of $\varprojlim\Bdrp_{,m} \gr{X_i}^{\widehat \otimes (n+1)}$ with respect to $J(n)$ (defined similarly as $J (0)$). The 
cosimplicial complex $\Sigma_{\textup{dR}}(\bullet)$ then computes the cohomology $R \Gamma_{\textup{inf}}(A/\Bdrp)$. 

Note that $\Bdrp \gr{X_i}$ is $\xi$-torsion free, and $J$ is finitely generated since $C\gr{X_i}$ is noetherian, hence $\Sigma_{\textup{dR}}$ (resp. $\Sigma_{\textup{dR}}(n)$) is flat over $\Bdrp$.\footnote{In particular $\Sigma_{\textup{dR}}$ is $\xi$-complete and $\xi$-torsion free. This also shows that $R \Gamma_{\textup{inf}} (A/\Bdrp)$ is derived $\xi$-adically complete.}  One then checks that the derived quotient $\Sigma_{\textup{dR}} (\bullet)/\xi$ is isomorphic to $\Sigma_C (\bullet)$, where $\Sigma_C(n)$ is the $J(n)$-adic completion of $C \gr{X_i}^{\widehat \otimes (n+1)}$. This agrees with the completion with respect to $\ker\big(C \gr{X_i}^{\widehat \otimes (n+1)} \twoheadrightarrow A \big)$, so the complex $\Sigma_C(\bullet)$ computes $R \Gamma_{\textup{inf}}(A/C)$. The lemma hence follows. 
\end{proof} 

The follow Proposition compares the infinitesimal cohomology with de Rham cohomology. The proof is similar to the proof of \cite{BdJ}. 
 
\begin{proposition} \label{prop:compare_C_inf_coh_with_dR} Let $A$ be a petit Tate algebra over $C$. There is a natural quasi-isomorphism between
$$R \Gamma_{\textup{inf}}(A/C) \cong \Omega_{A/C}^{\bullet}.$$ In particular, $R\Gamma_{\textup{inf}}(X/C) \cong R \Gamma_{\textup{dR}} (X/C)$, and 
therefore we have 
$$ R \Gamma_{\textup{inf}}(X/\Bdrp) \otimes^\L_{\Bdrp} C \isom  R \Gamma_{\textup{dR}} (X/C). $$
\end{proposition}

\begin{proof} 
As in the proof of the previous lemma, choose $C \gr{X_i} \twoheadrightarrow A$ and consider the complex $(\Sigma \ra \Sigma (1) \ra \Sigma(2) \ra \cdots)$ which computes $R \Gamma_{\textup{inf}}(A/C)$, here we use $\Sigma$ for $\Sigma_C$ to simplify notation. We will make a more refined choice of such a surjection later. Now consider the double complex
\[ 
\begin{tikzcd}[column sep = 1.5em, row sep = 1.2em]
\vdots &  \vdots &  \vdots &  \\
\Omega^2_{\Sigma} \arrow[r] \arrow[u]  & \Omega^2_{\Sigma(1)} \arrow[r] \arrow[u]   & \Omega^2_{\Sigma(2)} \arrow[r] \arrow[u]  & \cdots \\
\Omega^1_{\Sigma} \arrow[r] \arrow[u, "d"]   & \Omega^1_{\Sigma(1)} \arrow[r] \arrow[u, "d"]    & \Omega^1_{\Sigma(2)} \arrow[r] \arrow[u, "d"]    & \cdots   \\ 
\Sigma \arrow[r]  \arrow[u, "d"]   & \Sigma(1) \arrow[r] \arrow[u, "d"]   &  \Sigma(2) \arrow[r]  \arrow[u, "d"]   & \cdots
\end{tikzcd}
\]
where $\Omega^i_{\Sigma(n)}$ denotes the $J(n)$-adic completed differentials. For $i \ge 1$, each row $\Omega^i_\Sigma \ra \Omega_{\Sigma(1)}^i \ra \Omega_{\Sigma(2)}^i \ra \cdots$ is homotopic to $0$ (by the \cite{SP} 07L9, and note that being homotopic to $0$ is preserved under term-wise $p$-completions and $J(\bullet)$-completions). Moreover, for each of the maps $\Sigma \ra \Sigma(n)$, the induced morphism $\Omega^\bullet_{\Sigma} \ra \Omega^\bullet_{\Sigma(n)}$ is a quasi-isomorphism. This follows from the same proof of Lemma 2.13 of \cite{BdJ}, using the formal version of the Poincar\'e lemma (instead of the PD Poincar\'e lemma), and using the observation that each $\Sigma \ra \Sigma(n)$ induces an isomorphism between $\Sigma(n)$ and a formal power series $\Sigma [\![t_i]\!]$ over $\Sigma$. Now by comparing the two filtration spectral sequences associated to the double complex above, we have a natural quasi-isomorphism $R \Gamma_{\textup{inf}}(A/C) \cong \Omega_{\Sigma/C}^{\bullet}$. 

It remains to compare this with $\Omega_{A/C}^\bullet$. For this we make the choice of the map $C\gr{X_i} \twoheadrightarrow A$ as follows. By definition (of petit Tate algebras) there exists an $\etale$ map $f: C \gr{T_i} \ra A$, we enlarge the source by adding (possibly infinitely many) variables $Y_j$ to obtain a surjection $g$ as below. 
\[
\begin{tikzcd}[row sep = 1.5em]
C \gr{T_i}  \arrow[r] \arrow[d, swap, "f"] & C \gr{T_i, Y_j} \arrow[ld, two heads, "g"] \\
A
\end{tikzcd}
\]
Write $\beta_j = g(Y_j) \in A$. Let $\Sigma$ be the $\ker(g)$-completion of $C\gr{T_i, Y_j}$. Since $\Sigma$ is a pro-nilpotent thickening of $A$, and $C \gr{T_i}$ is $\etale$ (in the sense of \cite{Huber} 1.5.1), the map $C\gr{T_i} \ra \Sigma$ canonically lifts to a section $A \ra \Sigma$. This induces a map $A[Y_j] \ra \Sigma$, 
and thus induces a map $g: A [\![Y_j- \beta_j]\!] \ra \Sigma$ by passing to completions. Now these maps fit into the following commutative diagram (with only solid arrows) 
\[
\begin{tikzcd}[row sep = 1.5em]
 C\gr{X_i, Y_j} \arrow[d] \arrow[dr, "\sq f"]  \\
\Sigma \arrow[rd] \arrow[r, dashed] & A[\![Y_j - \beta_j]\!] \arrow[d] \\ 
& A 
\end{tikzcd} 
\]
where $\sq f$ is induced by $f$. Therefore by the universal property there exists a map $g': \Sigma \ra A[\![Y_j]\!]$ (the dashed arrow) such that $g \circ g'$ is the identity on $\Sigma$. Note that $g' \circ g $ is also identity by considering the image of $Y_j$, so we have $\Sigma \cong A[\![Y_j - \beta_j]\!]$, hence $\Omega_{\Sigma/C}^\bullet \isom \Omega_{A/C}^\bullet$. 
\end{proof}

\subsection{Comparison with $\Ainf$- and $\etale$ cohomology}
Next we return to the setup of the introduction, and compare the infinitesimal cohomology of $X$ with the crystalline cohomology over $\Acris$ (and hence the $\Ainf$-cohomology). We use this to obtain a $\Bdrp$-lattice in $H^i_{\ett}(X, \Z_p) \otimes \Bdr$.

\begin{lemma} \label{lemma:compare_inf_coh_with_abs_crys}
Let $\fX$ be a smooth formal scheme over $\spf \mO_C$. There is a functorial isomorphism $$b: R \Gamma_{\crys} (\fX_{\mO_C/p}/\Acris) \widehat \otimes^\L \Bdrp \isom R \Gamma_{\textup{inf}} (X/\Bdrp).$$
\end{lemma}

\begin{proof} 
Note that  $ R \Gamma_{\crys} (\fX /\Acris) \cong R \Gamma_{\crys} (\fX_{\mO_C/p}/\Acris)$. 
It suffices to restrict to affine opens $\spf R$ in $\fX_{\ett}$, and construct a map $R\Gamma_{\crys} (R/\Acris) \ra R\Gamma_{\textup{inf}}(R[\frac{1}{p}]/\Bdrp)$. For this we consider the following continuous functor 
$$\mu: \textup{CRIS}(R/\Acris)^{\textup{op}, \circ} \ra \textup{Inf}(R[\frac{1}{p}]/\Bdrp)^{\textup{op}, \circ}$$ from the (big) indiscrete crystalline site to the indiscrete infinitesimal site, 
which sends $ (B \twoheadrightarrow B/J) \longmapsto (\sq B \twoheadrightarrow (B/J) [\frac{1}{p}]).$
Here $\sq B$ is the completion of $B' := \varprojlim (B \otimes_{\Acris} \Bdrp_{, m})$ with respect to $\ker (B' \twoheadrightarrow B/J[\frac{1}{p}])$. The functor $\mu$ preserves fibre products and equalizers (both exist in the big crystalline site), hence the pullback functor $\mu_s$ is exact \footnote{see Tag 00WX, 00X4 and 00X5 of \cite{SP}.}. 
Thus $\mu$ induces a map $\textup{Shv} (\textup{CRIS}(R/\Acris)^{\textup{op}, \circ}) \ra \textup{Shv} (\textup{Inf}(R[\frac{1}{p}]/\Bdrp)^{\textup{op}, \circ})$ on topoi. The desired map $b$ follows from base change. To show that $b: R \Gamma_{\crys} (\fX/\Acris) \widehat \otimes^\L \Bdrp \isom R \Gamma_{\textup{inf}} (X/\Bdrp)$ is an isomorphism, it suffices to do so after reducing mod $\xi$ as both sides are derived $\xi$-complete. 
After mod $\xi$, the map constructed above becomes $R \Gamma_{\textup{dR}}(\fX) [\frac{1}{p}] \isom R \Gamma_{\textup{dR}}(X)$ by the de Rham comparison. 
\end{proof} 

In particular, in this setup we obtain a comparison between $\Ainf$-cohomology and the infinitesimal cohomology: 

\begin{corollary} \label{cor:Bdrp_lattice}
 Now assume that $\fX$ is smooth proper over $\spf \mO_C$. Then we have a functorial isomorphism 
$$ R \Gamma_{\Ainf} (\fX) \otimes^\L \Bdrp \isom R\Gamma_{\textup{inf}}(X/\Bdrp).$$
In particular, on cohomology groups we have $H^i_{\Ainf}(\fX) \otimes \Bdrp \cong H^i_{\textup{inf}}(X/\Bdrp)$. 
\end{corollary} 

\begin{proof} 
We combine Lemma \ref{lemma:invalid}  and Lemma \ref{lemma:compare_inf_coh_with_abs_crys} to obtain the composition $b \circ (h_{\crys} \otimes \Bdrp)^{-1}$. The rest follows from the fact that $H^i_{\Ainf}(\fX)$ is a Breuil--Kisin--Fargues module. 
\end{proof} 

\begin{corollary}
When $X = \fX_{C}^{\textup{ad}}$ is the generic fiber of $\fX$, the isomorphism above plus the $\etale$ comparison gives us a canonical isomorphism 
$$ H^i_{\textup{inf}}(X/\Bdrp) \otimes \Bdr \cong  H^i_{\ett} (X, \Z_p) \otimes \Bdr.$$ 
\end{corollary}

As a consequence we can reprove the following theorem of \cite{BMS}.

\begin{theorem}
\label{thm:BMS_on_recoverfromgeneric}
Suppose that $H^i_{\crys} (\fX_k/W(k))$ is torsion free, then we can recover $H^i_{\Ainf}(\fX)$ with its $\varphi$-action from the generic fiber $X$, more precisely from $H^i_{\ett} (X, \Z_p)$ and $H^i_{\textup{inf}}(X/\Bdrp)$. If moreover $H^{i+1}_{\crys}(\fX_k/W(k))$ is torsion free, then we can further recover the integral crystalline cohomology $H^i_{\crys}(\fX_k/W(k))$ with its $\varphi$-action. 
\end{theorem}

\begin{proof} 
If  $H^i_{\crys} (\fX_k/W(k))$ is torsion free, then by Remark \ref{remark:torsion} $H^i_{\Ainf}(\fX)$ is a free Breuil--Kisin--Fargues module, thus by Theorem \ref{thm:Fargues} it is determined by the pair $(T, \Xi)$, where $T = (H^i_{\Ainf} (\fX) \otimes W(C^\flat))^{\varphi = 1} 
\cong H^i_{\ett} (X, \Z_p)$, and $\Xi = H^i_{\Ainf} (\fX) \otimes \Bdrp \cong H^i_{\textup{inf}}(X/\Bdrp)$ by Corollary \ref{cor:Bdrp_lattice}. The first claim hence follows. Now suppose in addition that $H^{i+1}_{\crys} (\fX_k/W(k))$ is torsion free, then by Lemma \ref{lemma:relating_Tor_and_crystalline} we have a $\varphi$-equivariant isomorphism $H^i_{\crys}(\fX_k/W(k)) \cong H^i_{\Ainf} (\fX) \otimes W(k)$, hence the second claim follows from the first. 
\end{proof} 

\begin{remark} 
As explained in \cite{BMS}, by Kisin's construction (see Proposition 4.34 in \textit{loc.cit.}) the theorem above implies the version with Galois actions (when $\fX$ is defined over $\mO_K$ for some discretely valued extension $K/\Q_p$), namely part (2) of Corollary \ref{maincor:recover}.  
\end{remark}

\subsection{Compatibility of $h_{\crys}$ with filtration}
We first construct a map  
$$h_{\textup{dR}}: R \Gamma_{\textup{Inf}} (X/\Bdrp)  \longrightarrow R \Gamma_{\pet} (X, \mathbb{B}_{\textup{dR}}^+).$$ The construction is analogous to that of $h_{\crys}$, this time making use of pro-$\etale$ descent instead of flat/quasisyntomic descent. 

\begin{construction} \label{con:h_dR}
Let $X$ be a  smooth rigid analytic variety over $C$. It suffices to construct  
$h_{\textup{dR}}: R \Gamma_{\textup{Inf}} (A/\Bdrp)  \ra R \Gamma_{\pet} (\spa(A, A^\circ), \mathbb{B}_{\textup{dR}}^+)$ for petit affinoid Tate algebras $A$. Now as affinoid perfectoids form a basis for $\spa(A, A^\circ)_{\pet}$, it suffices to construct a functorial map from $R \Gamma_{\textup{Inf}} (A/\Bdrp)$ to $\mathbb{B}_{\textup{dR}}^+ (S, S^+)$, for all $U \ra \spa(A, A^\circ)_{\pet}$ affinoid perfectoid with $\hat U = \spa (S, S^+)$ (following notation in \cite{Scholze}, see also Theorem 6.5 in \textit{loc.cit.}). For this we observe that $\mathbb{B}_{\textup{dR}}^+ (S, S^+) \twoheadrightarrow S$ is an object in $\textup{Inf}(A/\Bdrp)$, and we get a functorial map $ \Gamma_{\textup{Inf}} (A/\Bdrp) \ra \mathbb{B}_{\textup{dR}}^+ (S, S^+)$
as in the construction of $h_{\crys}$. 
\end{construction}

Our final goal is to prove the filtration compatibility of $h_{\crys} \otimes \Bdr$. In order to do this, we first relate $h_{\crys} \otimes \Bdr$ to the map $h_{\textup{dR}}$ constructed above. Then, in the case when $\fX$ is base changed from $\fX_{\mO_K}$, we relate the latter map ($h_{\textup{dR}}$) to a map from $R \Gamma_{\pet} (X_{\cl K}, \Omega_{X_{\cl K}}^\bullet) \otimes^\L \Bdrp$ to  $ R \Gamma_{\pet} (X, \mathbb{B}_{\textup{dR}}^+)$ constructed in \cite{Scholze}, which is known to be a filtered isomorphism there. 

\begin{lemma} \label{lemma:base_change_compatible}
Let $\fX$ be a smooth proper formal scheme over $\mO_C$ with generic fiber $X$, then we have the following commutative diagram 
\[
\begin{tikzcd}[column sep = 1.2em, row sep = 1.5em]
R \Gamma_{\crys} (\fX/\Acris) \otimes^\L \Bdrp  \arrow[r] \arrow[d] 
& R \Gamma_{\Ainf} (\fX)   \otimes^\L \Bdrp \arrow[r] \arrow[d] & R \Gamma_{\Ainf} (\fX)  \otimes^\L \Bdr  \arrow[d] \\ 
R \Gamma_{\textup{Inf}} (X/\Bdrp) \arrow[r] 
& R \Gamma_{\pet} (X, \mathbb{B}_{\textup{dR}}^+) \arrow[r] &  
R \Gamma_{\pet} (X, \mathbb{B}_{\textup{dR}})
\end{tikzcd}
\]
where the left vertical map comes from the base change isomorphism. 
\end{lemma}

\begin{proof} 
It suffices to replace $\fX$ by $\spf R$ where $R$ is the $p$-completion of a smooth $\mO_C$-algebra. Moreover, by functoriality we further reduce to show the following: for a map $R \ra S$ from $R$ to a perfectoid $\mO_C$-algebra $S$ (the $S$ we have in mind is pro-$\etale$ but not necessarily quasisyntomic over $R$), the following diagram commutes.
\[
\begin{tikzcd}[row sep = 1.5em]
R \Gamma_{\crys}(R/\Acris) \arrow[r, "h_{\crys}"] \arrow[rd] & A \Omega_R \widehat \otimes^\L \Acris  \arrow[d] \\
& A \Omega_S \widehat \otimes^\L \Acris = \Acris (S)
\end{tikzcd}
\]
The vertical map comes from the natural map $A \Omega_R \ra A \Omega_S$, while the map $R\Gamma_{\crys} (R/\Acris) \ra \Acris(S)$ is induced from the crystalline site, by viewing the PD-thickening $\Acris (S) \ra S$ as an object in the crystalline site $\textup{CRIS} (R/\Acris)^{\textup{op}}$. The commutativity of the triangle is clear. 
\end{proof} 

Now we return the setup in Subsection \ref{ss:Bcris_comparison}, and let $\fX$ be the base change of $\fX_{\mO_K}$ to $\mO_C$. For a petit Tate algebra $A_K$ over $K$, we define $\textup{Inf}(A_K/K)$ similarly as $\textup{Inf}(A/C)$, and by the same proof of Proposition \ref{prop:compare_C_inf_coh_with_dR} (replacing $C$ by $K$) we have a functorial isomorphism $R \Gamma_{\textup{inf}}(X_K/K) \cong R \Gamma_{\textup{dR}}(X_K)$. The base change morphism 
for the indiscrete infinitesimal sites gives rise to a map   
$$R \Gamma_{\textup{inf}}(X_K/K) \otimes_K \Bdrp \longrightarrow R \Gamma_{\textup{inf}} (X_C/\Bdrp),$$
which is an isomorphism (again by considering its derived quotient mod $\xi$). Now we compose this isomorphism with $h_{\textup{dR}}$, and observe that from Lemma \ref{lemma:base_change_compatible}, in order to prove part (1) of Corollary \ref{maincor:recover} (the filtration compatibility of the $\Bcris$ comparison), it suffices to show that the following composition 
$$ R \Gamma_{\textup{dR}} (X_K) \otimes_K \Bdr \isom  R \Gamma_{\textup{inf}} (X_C/\Bdrp) \otimes \Bdr \xrightarrow{h_{\textup{dR}}}  R \Gamma_{\pet} (X_C, \mathbb{B}_{\textup{dR}})$$
is a filtered isomorphism. This follows from the following lemma (making use of Theorem 7.11 of \cite{Scholze}).


\begin{lemma} 
Retain the setup from above and adopt the notation from \cite{Scholze}. Then the following diagram commutes
\[
\begin{tikzcd} 
R \Gamma_{\textup{inf}}(X_K/K) \arrow[d, "\sim"{sloped, above}] \arrow[r] 
& R\Gamma_{\textup{inf}} (X_C/\Bdrp)  \arrow[r, "h_{\textup{dR}}"]  &  R \Gamma  (X_C, \mathbb{B}_{\textup{dR}}^+ )  \arrow[d, equal]  \\
 R\Gamma (X_{K}, \Omega_{X_{K}}^\bullet)    \arrow[r]  & R \Gamma  (X_C, \mO \mathbb{B}_{\textup{dR}}^+ \otimes_{\mO_{X_C}} \Omega_{X_{C}}^\bullet )  \arrow[r, "\sim"]  
 & R \Gamma  (X_C, \mathbb{B}_{\textup{dR}}^+ )
\end{tikzcd}
\]
where $\Omega_{X_{K}}^\bullet$ is the pullback of $\Omega_{X_{K}, \ett}^\bullet$ to the pro-$\etale$ site (left vertical isomorphism reflects the fact that the de Rham cohomology can be computed in the pro-$\etale$ site). The bottom arrows are constructed in \cite{Scholze}. 
\end{lemma}




\begin{proof} 
It suffices to assume that $X_K = \spa (R_K, R_K^\circ)$ for a very small Tate algebra $R_K$ over $K$. Recall that $R_K$ is very small if there are enough units $u_i, \beta_j \in (R_K^\circ)^\times$ that induces an $\etale$ map $K \gr{T_i^{\pm}} \ra R_K$ and a surjective map $K \gr{T_i^{\pm}, Y_j^{\pm}} \twoheadrightarrow R_K$, with $T_i \mapsto u_i, Y_j \mapsto \beta_j$ (compare with petit Tate algebras). The pro-$\etale$ cover $K \gr{T_i^\pm} \ra K \gr{T_i^{\pm 1/p^\infty}}$ induces a pro-$\etale$ cover $R_K \ra R_{K, \infty}$ (resp. a pro-$\etale$ cover $R_C \ra R_\infty$ via base change). It suffices to check the commutativity of the diagram after passing to the pro-$\etale$ cover for the horizontal arrows on the bottom of the diagram. In other words, we want to show that the following diagram commutes
\[
\begin{tikzcd}
 R\Gamma_{\textup{inf}}(R_K/K) \arrow[r] \arrow[d, equal]  &  R\Gamma_{\textup{inf}}(X_C/\Bdrp) \arrow[r, "h_{\textup{dR}, \infty}"] &  \mathbb{B}_{\textup{dR}}^+ (R_\infty) \arrow[d, equal] \\
\Omega_{R_K}^\bullet  \arrow[r]  &  \Omega_{R_\infty}^\bullet \longrightarrow  \mO \mathbb{B}_{\textup{dR}}^+ (R_\infty) \otimes_{R_\infty} \Omega_{R_\infty}^\bullet  \arrow[r, "\sim"] & \mathbb{B}_{\textup{dR}}^+ (R_\infty)
\end{tikzcd}
\]
where $h_{\textup{dR}, \infty}$ is constructed similarly as $h_{\textup{dR}}$, by viewing $\Bdrp(R_\infty) \ra R_\infty$ as a pro-nilpotent thickening, and the last isomorphism in the bottom arrow comes from $ \mathbb{B}_{\textup{dR}}^+$-Poincar\'e Lemma of Scholze (Corollary 6.13 in \textit{loc.cit.}).  

As in the proof of Proposition \ref{prop:compare_C_inf_coh_with_dR} (replacing $K \gr{T_i, Y_j}$ by $K \gr{T_i^{\pm}, Y_j^{\pm}}$), let $\Sigma_K$ be the completion of $K \gr{T_i^{\pm}, Y_j^{\pm}}$ along the kernel of its projection to $R_K$. Again we have an isomorphism $\Sigma_K \isom R_K [\![Y_j - \beta_j ]\!]$ (note that there is indeed a map from $K \gr{T_i^{\pm}, Y_j^{\pm}} \ra R_K [\![Y_j - \beta_j]\!]$ because $Y_j$ is invertible in the latter power series ring, as each $\beta_j$ is a unit). This describes the map  $\Omega_{R_K}^\bullet \isom \Omega_{R_K [\![Y_j - \beta_j]\!]}^\bullet \cong R \Gamma_{\textup{inf}}(R_K/K)$  on the left.  The  top horizontal arrow is given by  lifting the map $K \gr{T_i^{\pm}, Y_j^{\pm}} \twoheadrightarrow R_K \ra R_\infty$ along $\mathbb{B}_{\textup{dR}}^+ (R_\infty) \ra R_\infty$, sending $T_i$ to $[T_i^\flat] \in \Ainf(R_\infty)$, where $T_i^\flat = (T_i, T_i^{1/p}, T_i^{1/p^2}, ... ) \in R_\infty^\flat$. This in particular determines a unique lift $R_K \ra \mathbb{B}_{\textup{dR}}^+ (R_\infty)$. Now unwinding definitions, the top horizontal arrows composed with the left vertical isomorphism can be described by $\Omega_{R_K}^\bullet \ra R_K \ra \mathbb{B}_{\textup{dR}}^+ (R_\infty)$. The bottom horizontal map factors through 
$$R_K \longrightarrow \mO \mathbb{B}_{\textup{dR}}^+ (R_\infty) \isom \mathbb{B}_{\textup{dR}}^+ (R_\infty)[\![X_i]\!] \xrightarrow{\textup{proj}} \mathbb{B}_{\textup{dR}}^+ (R_\infty),$$
where the isomorphism in the middle is described by Proposition 6.10 in \textit{loc.cit}, which sends $T_i \otimes 1 - 1 \otimes [T_i^\flat] \mapsto X_i$, and the projection map induces a quasi-isomorphism on the de Rham complexes.  Now the commutativity of the previous diagram follows from the fact that both maps agree on $T_i$: as the bottom morphism sends $T_i \mapsto (X_i + 1\otimes [T_i^\flat]) \mapsto [T_i^\flat]$. This concludes the lemma. 
\end{proof} 

Finally, this finishes the proof of the $\Bcris$ comparison, which we restate. 

\begin{theorem}
Suppose that $\fX_{\mO_K}$ is a smooth proper formal scheme over $\mO_K$, with $K$ a discrete $p$-adic field as before, with residue field $k_0$. Let $C = C_K$. Then $h_{\crys}$ induces an isomorphism 
$$ H^i_{\textup{crys}} (\fX_{k_0}/W(k_0)) \otimes_{W(k_0)} \Bcris \isom H^i_{\ett} (\fX_{C}^{\textup{ad}}, \Z_p) \otimes_{\Z_p} \Bcris$$
compatible with the $\textup{Gal}_K$-action, Frobenius, and filtration. 
\end{theorem}


\bibliographystyle{alpha}
\bibliography{dRW}

\end{document}